\documentclass[11pt,reqno]{amsart}
\usepackage[T1]{fontenc}
\usepackage{graphicx}
\usepackage{cite}
\usepackage[colorlinks=true,linkcolor=MyLinkColor,citecolor=MyLinkColor]{hyperref}

\usepackage{color}
\definecolor{MyLinkColor}{rgb}{0,0,0.4}
\newtheorem{theorem}{Theorem}[section]
\newtheorem{corollary}[theorem]{Corollary}
\newtheorem{lemma}[theorem]{Lemma}
\newtheorem{proposition}[theorem]{Proposition}

\numberwithin{equation}{section}

\begin{document}
\title[Bounded weak solutions to the thin film Muskat problem]{Bounded weak solutions to the thin film Muskat problem via an infinite family of Liapunov functionals}
\thanks{}
\author{Philippe Lauren\c{c}ot}
\address{Institut de Math\'ematiques de Toulouse, UMR~5219, Universit\'e de Toulouse, CNRS \\ F--31062 Toulouse Cedex 9, France}
\email{laurenco@math.univ-toulouse.fr}
\author{Bogdan-Vasile Matioc}
\address{Fakult\"at f\"ur Mathematik, Universit\"at Regensburg \\ D--93040 Regensburg, Deutschland}
\email{bogdan.matioc@ur.de}


\date{\today}

\begin{abstract}
A countably infinite family of Liapunov functionals is constructed for the thin film Muskat problem, which is a second-order degenerate parabolic system featuring cross-diffusion. More precisely, for each $n\geq 2$ we construct an homogeneous polynomial of degree $n$, which  is convex on $[0,\infty)^2$, with the property that its integral is a Liapunov functional for the problem. Existence of global bounded non-negative weak solutions is then shown in one space dimension.
\end{abstract}

\maketitle

%
%
\pagestyle{myheadings}
\markboth{\sc{Ph.~Lauren\c cot \& B.-V.~Matioc}}{\sc{Bounded weak solutions to the thin film Muskat problem}}

\section{Introduction}\label{sec01}

The thin film Muskat problem describes the dynamics of the respective heights of two immiscible fluids with different densities $(\rho_-,\rho_+)$ and viscosities $(\mu_-,\mu_+)$ in a porous media and reads
\begin{subequations}\label{tfm1}
\begin{align}
	\partial_t f & = \mathrm{div}\left(f \nabla\left[ (1+R) f + R g \right] \right) \;\text{ in }\; (0,\infty)\times \Omega\,, \label{tfm1a} \\
	\partial_t g & = \mu R\, \mathrm{div}\left[ g\nabla (f + g) \right] \;\text{ in }\; (0,\infty)\times \Omega\,, \label{tfm1b}
\end{align}
supplemented with homogeneous Neumann boundary conditions
\begin{equation}
	\nabla f\cdot \mathbf{n} = \nabla g\cdot \mathbf{n} = 0 \;\text{ on }\; (0,\infty)\times \partial\Omega\,, \label{tfm1c}
\end{equation}
and non-negative initial conditions
\begin{equation}
	(f,g)(0) = (f^{in},g^{in}) \;\text{ in }\; \Omega\,. \label{tfm1d}
\end{equation}
\end{subequations}
In \eqref{tfm1}, $\Omega$ is a bounded domain of $\mathbb{R}^d$, $d\ge 1$, with smooth boundary $\partial\Omega$, $f$ and $g$ denote the heights of the heavier and lighter fluids, 
respectively, and $R:=\rho_+/(\rho_- - \rho_+)$ and $\mu:=\mu_-/\mu_+$ are positive parameters depending solely on the densities ($\rho_- > \rho_+$) and viscosities of the two fluids. We recall that \eqref{tfm1} can be derived from the classical Muskat problem by a lubrication approximation \cite{EMM2012, JM2014, WM2000}. 

From a mathematical viewpoint, the system~\eqref{tfm1} is a quasilinear degenerate parabolic system with full diffusion matrix and it is well-known that the analysis of cross-diffusion systems is in general rather involved. Fortunately, as noticed in \cite{EMM2012}, an important property of \eqref{tfm1} is the availability of an energy functional. Specifically, 
\begin{equation}
	\mathcal{E}(f,g) := \frac{1}{2} \int_\Omega \left[ f^2 + R (f+g)^2 \right]\ \mathrm{d}x \label{in1}
\end{equation}
is a non-increasing function of time along the trajectories of \eqref{tfm1}. 
In fact, a salient feature of~\eqref{tfm1}, first observed in~\cite{LM2013}, is that it has, at least formally, a gradient flow structure for the energy~$\mathcal{E}$ with respect to the $2$-Wasserstein distance. 
This structure provides in particular a variational scheme to establish the existence of weak solutions to \eqref{tfm1}. 
Furthermore, as first noticed in \cite{ELM2011} and subsequently used in \cite{ACCL2019, LM2013}, the entropy functional
\begin{equation}
	\mathcal{H}(f,g) := \int_\Omega \left[ L(f) + \frac{1}{\mu} L(g) \right]\ \mathrm{d}x\,, \label{in2}
\end{equation}
with $L(r) := r\ln{r}-r+1\ge 0$, $r\ge 0$, is also a non-increasing function of time along the trajectories of \eqref{tfm1}. 

Building upon the above mentioned properties, the existence of non-negative global weak solutions~${(f,g)}$ to~\eqref{tfm1} satisfying
\begin{equation}
	(f,g) \in L_\infty((0,T),L_2(\Omega,\mathbb{R}^2))\cap L_2((0,T),H^1(\Omega,\mathbb{R}^2))\,, \label{in3}
\end{equation}
is shown in \cite{ELM2011, LM2017} in one space dimension $d=1$, see also \cite{AIJM2018, BGB2019} for the existence of global weak solutions to a generalized version of \eqref{tfm1} in the $d$-dimensional torus with periodic boundary conditions instead of the homogeneous Neumann boundary conditions~\eqref{tfm1c}.
 Local existence and uniqueness of classical solutions to \eqref{tfm1} with positive initial data are reported in \cite{EMM2012}, along with the local stability of spatially uniform steady states. 
 As for the Cauchy problem when $\Omega=\mathbb{R}^d$ and~${d\in \{1,2\}}$, non-negative global weak solutions are constructed in \cite{ACCL2019,  LM2013}
  for non-negative initial conditions~${(f^{in},g^{in})\in L_1(\mathbb{R}^d,\mathbb{R}^2)\cap L_2(\mathbb{R}^d,\mathbb{R}^2)}$ with finite second moments, exploiting the aforementioned gradient flow structure to set up a variational scheme, see also \cite{La2017} for an extension to a multicomponent version of~\eqref{tfm1} in one space dimension. 
  This approach is further developed in \cite{ZM2015} to investigate the existence of non-negative global weak solutions to a broader class of quasilinear cross-diffusion systems.

In view of \eqref{in3}, weak solutions to \eqref{tfm1} have rather low regularity. 
It is actually a general feature of cross-diffusion systems that classical regularity is hard to reach. In particular, the cross-diffusion structure impedes the use of bootstrap arguments which have proved efficient for triangular systems.
 A different route to improved regularity is to look for additional estimates and the purpose of this paper is to derive (formally) $L_n$-estimates for solutions to~\eqref{tfm1} for all integers $n\ge 2$, 
 eventually leading to $L_\infty$-estimates in the limit $n\to\infty$. 
 This feature paves the way to the construction of non-negative global bounded weak solutions to~\eqref{tfm1} but, as explained below,  
 we are only able to complete this construction successfully in one space dimension $d=1$.
  The first main contribution of this paper is actually to show that, for each $n\ge 2$, there is an homogeneous polynomial $\Phi_n$ of degree~$n$, which is non-negative and convex on $[0,\infty)^2$, and such that
\begin{equation}
	u=(f,g) \longmapsto \int_\Omega \Phi_n(u)\ \mathrm{d}x
\end{equation}
is a Liapunov functional for \eqref{tfm1}. 
More precisely, the first main result of this paper is the following.

\begin{theorem}\label{MainTh}
	Let $R>0$, $\mu>0$, and $u^{in} := (f^{in},g^{in})\in L_{\infty,+}(\Omega,\mathbb{R}^2)$. 
	If $u=(f,g)$ is a sufficiently regular  solution to \eqref{tfm1} on $[0,\infty)$ with non-negative components, then
	\begin{itemize}
	 \item[(l1)] for all $t\ge 0$, 
	\begin{equation}
		\int_\Omega \Phi_1(u(t,x))\ \mathrm{d}x + \int_0^t \int_\Omega \left[ |\nabla f|^2 + R |\nabla (f+g)|^2 \right](s,x)\ \mathrm{d}x\mathrm{d}s \le \int_\Omega \Phi_1(u^{in}(x))\ \mathrm{d}x\,, \label{p3}
	\end{equation}
	where $\Phi_1(X) := L(X_1) + L(X_2)/\mu$, $X=(X_1,X_2)\in [0,\infty)^2$;\medskip
	\item[(l2)] for all $n\ge 2$ and all $t\ge 0$, 
	\begin{subequations}\label{p4}
		\begin{equation}
			\int_\Omega \Phi_n(u(t,x))\ \mathrm{d}x \le \int_\Omega \Phi_n(u^{in}(x))\ \mathrm{d}x\,, \label{p4a}
		\end{equation}
		where $\Phi_n$ is the homogeneous polynomial of degree $n$ given by
		\begin{equation}
			\Phi_n(X) := \sum_{j=0}^n a_{j,n} X_1^j X_2^{n-j}\,, \qquad X=(X_1,X_2)\in \mathbb{R}^2\,, \label{p4b}
		\end{equation}
		with $a_{0,n}:=1$, 
		\begin{equation} \label{p4c}
			\begin{split}
				a_{j,n} & := \binom{n}{j} \prod_{k=0}^{j-1} \frac{k +\alpha_{k,n}}{\alpha_{k,n}} > 0\,, \qquad  1\le j \le n\,, \\
				\alpha_{k,n} & := R [ k + \mu(n-k-1)]> 0\,, \qquad 0 \le k \le n-1\,.
			\end{split}
		\end{equation}
	\end{subequations}
	In addition, $\Phi_n$ is convex on $[0,\infty)^2$;\medskip
	\item[(l3)] for $t\ge 0$,
	\begin{equation}
		\|(f+g)(t)\|_\infty \le \frac{1+R}{R} \|f^{in}+g^{in}\|_\infty\,. \label{p5}
	\end{equation} 
	\end{itemize}
\end{theorem}

For $n=2$, Theorem~\ref{MainTh} gives $(a_{0,2},a_{1,2},a_{2,2})=(1,2,(1+R)/R)$. Therefore,
\begin{equation*}
	\Phi_2(X) = \frac{1+R}{R} X_1^2 + (X_1+X_2)^2\,, \qquad  X\in\mathbb{R}^2\,,
\end{equation*}
and 
\begin{equation*}
	\mathcal{E}(f,g) = \frac{R}{2} \int_\Omega \Phi_2((f,g))\ \mathrm{d}x\,,
\end{equation*}
so that we recover the time monotonicity of the energy from \eqref{p4a} with $n=2$.

\medskip

It seems worth pointing out that the availability of an infinite number of Liapunov functionals, 
eventually leading to $L_\infty$-estimates, seems rather seldom for cross-diffusion systems and
 that we are not aware of other systems sharing this feature. 
 Whether it is possible to extend the analysis performed in this paper to a broader class of cross-diffusion systems will be the subject of future research.

\medskip

To construct the family of polynomials $(\Phi_n)_{n\ge 2}$, we introduce $u=(f,g)$ and the mobility matrix 
\begin{equation}\label{moma}
	M(X) = (m_{jk}(X))_{1\le j,k\le 2} := 
	\begin{pmatrix}
		(1+R) X_1 & R X_1 \\
		\mu R X_2 & \mu R X_2
	\end{pmatrix}\,, \qquad X\in\mathbb{R}^2\,,
\end{equation}
so that \eqref{tfm1a}-\eqref{tfm1b} becomes
\begin{equation}
	\partial_t u - \sum_{i=1}^d \partial_i \left( M(u) \partial_i u \right) = 0 \;\text{ in }\; (0,\infty)\times \Omega\,. \label{in4}
\end{equation}
We then use the observation that, since $\Phi\in C^2(\mathbb{R}^2)$, \eqref{tfm1c}, \eqref{in4}, and the symmetry of the Hessian matrix $D^2\Phi(u)$ of $\Phi$ entail that
\begin{equation}
	\frac{d}{dt} \int_\Omega \Phi(u)\ \mathrm{d}x + \sum_{i=1}^d \int_\Omega \langle D^2\Phi(u) M(u) \partial_i u , \partial_i u \rangle\ \mathrm{d}x = 0\,. \label{in5}
\end{equation} 
It readily follows from \eqref{in5} that $\Phi$ provides a Liapunov functional for \eqref{in4} as soon 
as the matrix~${D^2\Phi(u) M(u)}$ is  symmetric and  positive semidefinite. 
Using the ansatz~\eqref{p4b} for $\Phi=\Phi_n$ and the explicit form of the matrix $M$, we then compute $D^2\Phi_n(u) M(u)$ and require that it is a symmetric matrix, thereby obtaining \eqref{p4c}. Direct computations then show that the polynomial thus obtained is actually non-negative and convex on $[0,\infty)^2$, see section~\ref{sec04}  and appendix~\ref{secapA}.

\medskip

Having uncovered the estimates~(l1)-(l3) at a somewhat formal level, it is tempting to exploit them to construct a bounded weak solution to \eqref{tfm1} endowed with these properties. The difficulty we face here is the construction of an appropriate approximation of \eqref{tfm1} having sufficiently smooth solutions for which the estimates~(l1)-(l3) remain valid. In particular, boundedness of solutions to the approximation seems to be required to be able to compute the integral of $\Phi_n(u)$. Unfortunately, we have yet been unable to design an approximation scheme producing bounded solutions while preserving the structure leading to~(l1)-(l3) that could work in arbitrary space dimensions and we only provide below the existence of a bounded weak solution to~\eqref{tfm1} in one space dimension $d=1$. 

\begin{theorem}\label{ThBWS}
	Let $R>0$, $\mu>0$, $u^{in} := (f^{in},g^{in})\in L_{\infty,+}(\Omega,\mathbb{R}^2)$, and assume that $d=1$ (so that~$\Omega$ is a bounded interval of $\mathbb{R}$).
	 Then, there is a bounded weak solution $u:=(f,g)$ to \eqref{tfm1} which satisfies:
	\begin{itemize}
	\item[(p1)] for each $T>0$, 
	\begin{equation}
		(f,g)\in L_{\infty,+}((0,T)\times \Omega,\mathbb{R}^2) \cap L_2((0,T),H^1(\Omega,\mathbb{R}^2))\cap W_2^1((0,T),H^1 (\Omega,\mathbb{R}^2)')\,; \label{p1}
	\end{equation}

	\item[(p2)] for all $\varphi\in H^1(\Omega)$ and $t\ge 0$,
	\begin{subequations}\label{p2}
		\begin{equation}
			\int_\Omega (f(t,x)-f^{in}(x)) \varphi(x)\ \mathrm{d}x  +\int_0^t \int_\Omega f(s,x) \partial_x\left[ (1+R) f + R g \right](s,x) \cdot \partial_x\varphi(x)\, \mathrm{d}x\mathrm{d}s =0 \label{p2a}
		\end{equation} 
		and
		\begin{equation}
			\int_\Omega (g(t,x)-g^{in}(x)) \varphi(x)\ \mathrm{d}x + \mu R \int_0^t \int_\Omega g(s,x) \partial_x\left( f + g \right)(s,x) \cdot \partial_x\varphi(x)\, \mathrm{d}x\mathrm{d}s =0\,; \label{p2b}
		\end{equation}
	\end{subequations} 
	
\item[(p3)] and the properties (l1), (l2), and (l3) stated in Theorem~\ref{MainTh}.
\end{itemize}
\end{theorem}

A key ingredient in the proof of Theorem~\ref{ThBWS} is the continuous embedding of $H^1(\Omega)$ in $L_\infty(\Omega)$, which readily provides the above mentioned boundedness of solutions to the approximation of~\eqref{tfm1} designed below and is of course only available in one space dimension. Besides, we employ a rather classical approximation approach, relying on a time implicit Euler scheme 
with constant time step~${\tau\in (0,1)}$ for the time discretization and a suitable modification of the mobility matrix to a non-degenerate one.

As a consequence of Theorem~\ref{ThBWS} and of the estimate~\eqref{LUB}, we obtain uniform $L_\infty$-bounds for the height $f$ of the denser fluid in the regime where $R\to 0$ and $\mu$ is fixed. Such an estimate has been used recently in \cite[Corollary~1.4]{LM2021a} when performing the singular limit $R\to0$ (while~$\mu$ is kept fixed or~$\mu\to\infty$) in the thin film Muskat problem \eqref{tfm1} in order to recover the porous medium equation $\partial_t f=\mathrm{div}\big(f \nabla f\big)$ in the limit.

\begin{corollary}\label{Cor:1}
If $R\max\{1,\mu\}\in (0,1/(2e)]$, then the solution $u=(f,g)$ to \eqref{tfm1} from Theorem~\ref{ThBWS} satisfies
\begin{equation}\label{estaaa}
\|f(t)\|_\infty\leq  \big( 1 +  e\max\{1,\mu\} \big) \|f^{in}\|_\infty + \| g^{in}\|_\infty\,, \qquad t\geq0.
\end{equation}
\end{corollary}

\medskip

We provide the proof of Theorem~\ref{ThBWS} in sections~\ref{sec02} and ~\ref{sec03} below. It involves three steps: we begin with the existence of a weak solution to the implicit time discrete scheme associated to \eqref{tfm1} which satisfies a time discrete version of the estimates~\eqref{p4} and~\eqref{p5} (section~\ref{sec02}).
 This result is achieved with an approximation procedure which is designed and studied in section~\ref{sec02.1}, a technical lemma being postponed to appendix~\ref{secapB}. 
 The next section~\ref{sec02.2} is devoted to the time discrete version of the estimates~\eqref{p4} and~\eqref{p5}, the proof of the properties of the polynomials $\Phi_n$, $n\ge 2$, 
 being collected in appendix~\ref{secapA}. 
 The proof of Theorem~\ref{ThBWS} is given in section~\ref{sec03} and is based on a compactness method. We finally prove Theorem~\ref{MainTh} in section~\ref{sec04}.

\bigskip

\paragraph{\textbf{Notation.}} For $p\in [1,\infty]$, we denote by $\|\cdot\|_p$ the $L_p$-norm in $L_p(\Omega)$, $L_p(\Omega,\mathbb{R}^2) := L_p(\Omega)\times L_p(\Omega)$, and~$H^1(\Omega,\mathbb{R}^2):= H^1(\Omega)\times H^1(\Omega)$. The positive cone of a Banach lattice $E$ is denoted by $E_+$.
 Next,~${\mathbf{M}_2(\mathbb{R})}$ denotes the space of $2\times 2$ real-valued matrices, ${\mathbf{Sym}_2(\mathbb{R})}$ is the subset of ${\mathbf{M}_2(\mathbb{R})}$ consisting of symmetric matrices, and $\mathbf{SPD}_2(\mathbb{R})$ is the set of symmetric and positive definite matrices in $\mathbf{M}_2(\mathbb{R})$. Finally, we denote the positive part of a real number $r\in\mathbb{R}$ by~${r_+:=\max\{r,0\}}$ and $\langle \cdot,\cdot \rangle$ is the scalar product on $\mathbb{R}^2$.

\section{A time discrete scheme: $d=1$}\label{sec02}

Throughout this section, we assume that $d=1$ and $\Omega$ is a bounded open interval of $\mathbb{R}$.
In order to construct bounded non-negative global  weak solutions to the evolution problem~\eqref{tfm1} we introduce an implicit time discrete scheme, see \eqref{ex1a}-\eqref{ex1b}, 
as well as a regularized version of this scheme, see \eqref{ex12a}. The aim of this section is to prove the existence of a bounded weak solution to the implicit time discrete scheme, as stated in Proposition~\ref{prop.e1}.

\begin{proposition}\label{prop.e1}
Given  $\tau>0$ and $U=(F,G)\in L_{\infty,+}(\Omega,\mathbb{R}^2)$, there is a weak solution~$u=(f,g)$  with~$u\in H^1_+(\Omega,\mathbb{R}^2)$ to 
\begin{subequations}\label{ex1}
\begin{align}
	\int_\Omega \Big[ f \varphi + \tau f \partial_x\left[ (1+R) f + R g \right] \cdot\partial_x\varphi \Big]\ \mathrm{d}x & = \int_\Omega F \varphi\ \mathrm{d}x\,, \qquad \varphi\in H^1(\Omega)\,, \label{ex1a} \\
	\int_\Omega \Big[ g \psi + \tau \mu R g \partial_x(f + g)\cdot \partial_x\psi \Big]\ \mathrm{d}x & = \int_\Omega G \psi\ \mathrm{d}x\,, \qquad \psi\in H^1(\Omega)\,, \label{ex1b}
\end{align}
\end{subequations}
which also satisfies
\begin{equation}
	\int_\Omega \Phi_n(u)\ \mathrm{d}x \le \int_\Omega \Phi_n(U)\ \mathrm{d}x \label{ex2}
\end{equation}
for $n\ge 2$,
\begin{equation}
		\|f+g\|_\infty \le \frac{1+R}{R} \|F+G\|_\infty\,, \label{ex2a}
\end{equation}
 and
\begin{equation}
	\int_\Omega \Phi_1(u)\ \mathrm{d}x + \tau \int_\Omega \left[  |\partial_x f|^2 + R |\partial_x (f + g)|^2 \right]\ \mathrm{d}x \le \int_\Omega \Phi_1(U)\ \mathrm{d}x\,. \label{ex2b}
\end{equation}
\end{proposition}

We fix $\tau>0$ and $U=(F,G)\in L_{\infty,+}(\Omega,\mathbb{R}^2)$.
 Recalling the definition~\eqref{moma} of the mobility matrix 
 \begin{equation*}
 	M(X) = 
 	\begin{pmatrix}
 		(1+R) X_1 & R X_1 \\
 		\mu R X_2 & \mu R X_2
 	\end{pmatrix}\,, \qquad X\in\mathbb{R}^2\,,
 \end{equation*}
an alternative formulation of \eqref{ex1} reads
\begin{equation}
	\int_\Omega \left[\langle u,v \rangle + \tau \langle M(u) \partial_x u , \partial_x v\rangle \right]\ \mathrm{d}x = \int_\Omega \langle U,v\rangle\ \mathrm{d}x\,, \qquad v\in H^1(\Omega,\mathbb{R}^2)\,. \label{ex3}
\end{equation}

Obviously, the mobility matrix $M(X)$ is in general not symmetric and the associated quadratic form 
\begin{equation*}
\mathbb{R}^2\ni\xi=(\xi_1,\xi_2)  \mapsto \sum_{j,k=1}^2 m_{jk}(X) \xi_j \xi_k\in \mathbb{R}
\end{equation*}
is not positive definite (even if $X\in[0,\infty)^2)$, two features which complicate the analysis concerning the solvability of \eqref{ex3}. Fortunately, as noticed in \cite{DGJ1997}, the underlying gradient flow structure provides a way to transform \eqref{ex3} to an elliptic system with symmetric and positive semidefinite matrix. 
 More precisely, we introduce the symmetric matrix $S$ with constant coefficients 
\begin{equation*}
	S := \begin{pmatrix}
		1+R & R \\
		R & R
	\end{pmatrix}\,,
\end{equation*} 
which is actually the Hessian matrix of $R\Phi_2/2$. Clearly, $S$ belongs to $\mathbf{SPD}_2(\mathbb{R})$ and
\begin{equation}
	\langle S \xi , \xi \rangle = \xi_1^2 + R \left( \xi_1 + \xi_2 \right)^2 \ge \frac{R}{1+2R} |\xi|^2\,, \qquad \xi\in\mathbb{R}^2\,. \label{ex4}
\end{equation}
Choosing $Sv$ instead of $v\in H^1(\Omega,\mathbb{R}^2)$ as a test function in \eqref{ex3} and using the symmetry of $S$,
 lead to another alternative formulation of \eqref{ex1a}-\eqref{ex1b} (or \eqref{ex3}), which reads
\begin{equation}
	\int_\Omega \left[ \langle Su,v \rangle + \tau \langle SM(u) \partial_x u , \partial_x v\rangle \right]\ \mathrm{d}x 
	= \int_\Omega \langle SU,v\rangle\ \mathrm{d}x\,, \qquad v\in H^1(\Omega,\mathbb{R}^2)\,. \label{ex5}
\end{equation}
We next observe that, for $X\in [0,\infty)^2$,
\begin{equation}
	SM(X) = \begin{pmatrix}
		(1+R)^2 X_1 + \mu R^2 X_2 & (1+R)R X_1 + \mu R^2 X_2 \\
		(1+R)R X_1 + \mu R^2 X_2 & R^2 X_1 + \mu R^2 X_2
	\end{pmatrix} \label{ex6a}
\end{equation}
and
\begin{equation}
	\begin{split}
	\langle SM(X)\xi,\xi\rangle & = X_1 ((1+R) \xi_1 + R \xi_2 )^2 + \mu R^2 X_2 (\xi_1+\xi_2)^2\geq 0\,.
	\end{split}\label{ex6b}
\end{equation}
Consequently, $SM(X)$ belongs to $\mathbf{SPD}_2(\mathbb{R})$ for all $X\in (0,\infty)^2$ and the formulation~\eqref{ex5} seems more appropriate to study the
 solvability of \eqref{ex1a}-\eqref{ex1b}. However, the matrix $SM(X)$ is still degenerate as $X_1\to 0$ or $X_2\to 0$, so that we first solve a regularized problem in the next section.  

\subsection{A regularization of the time discrete scheme}\label{sec02.1}

Let $\varepsilon\in (0,1)$ and define
\begin{equation}
	M_{\varepsilon}(X)  := (m_{\varepsilon, jk}(X))_{1\le j,k\le 2}:= \varepsilon I_2 + M((X_{1,+},X_{2,+}))\,, \qquad X\in\mathbb{R}^2\,. \label{ex10}
\end{equation}

\begin{lemma}\label{lem.ex2}
Given $\tau>0$, $U=(U_1,U_2)\in L_{\infty,+}(\Omega,\mathbb{R}^2)$, and $\varepsilon\in (0,1)$, there is a weak solution~${u_{\varepsilon} =(u_{1,\varepsilon},u_{2,\varepsilon})\in H^1_+(\Omega,\mathbb{R}^2)}$ to 
\begin{equation}
	\int_\Omega \left[ \langle u_{\varepsilon} , v \rangle + \tau \langle M_{\varepsilon}(u_{\varepsilon}) \partial_x u_{\varepsilon} , \partial_x v \rangle \right]\ \mathrm{d}x = \int_\Omega \langle U , v \rangle\ \mathrm{d}x\,, \qquad v\in H^1(\Omega,\mathbb{R}^2)\,. \label{ex12a}
\end{equation}
Moreover, 
\begin{equation}
	\| u_{1,\varepsilon} + u_{2,\varepsilon} \|_\infty \le \frac{1+R}{R} \|U_1+U_2\|_\infty\,, \label{ex12c}
\end{equation}
and, for $n\ge 2$,
\begin{equation}
	\int_\Omega \Phi_n(u_{\varepsilon})\ \mathrm{d}x \le \int_\Omega \Phi_n(U)\ \mathrm{d}x\,. \label{ex12b}
\end{equation}
\end{lemma}

\begin{proof}
For each $\varepsilon\in (0,1)$, $M_{\varepsilon}$ lies in $C(\mathbb{R}^2,\mathbf{M}_2(\mathbb{R}))$ and satisfies 
\begin{subequations}\label{ex11}
\begin{equation}
	\begin{split}
		m_{\varepsilon,11}(X) \ge m_{\varepsilon,12}(X) & = 0\,, \qquad X\in (-\infty,0)\times \mathbb{R}\,, \\
		m_{\varepsilon,22}(X) \ge m_{\varepsilon,21}(X) & = 0\,, \qquad X\in \mathbb{R}\times (-\infty,0)\,.
	\end{split} \label{ex11b} 
\end{equation}
In addition, it follows from \eqref{ex4}, \eqref{ex6a}, \eqref{ex6b}, and \eqref{ex10} that $SM_{\varepsilon}(X)$ belongs to $\mathbf{SPD}_2(\mathbb{R})$ for all~${X\in\mathbb{R}^2}$ with
\begin{equation}
	\langle SM_{\varepsilon}(X)\xi,\xi\rangle \ge \frac{\varepsilon R}{1+2R} |\xi|^2\,, \qquad \xi\in\mathbb{R}^2\,. \label{ex11c}
\end{equation}
\end{subequations}
According to the properties~\eqref{ex11}, we are now in a position to apply Lemma~\ref{lem.ap1} (with $A=S$ and~${B=M_{\varepsilon}}$) and deduce that there is a non-negative solution $u_{\varepsilon}\in H^1(\Omega,\mathbb{R}^2)$ to \eqref{ex12a}.

In the remaining part, we prove that~$u_\varepsilon$  satisfies  both estimates~\eqref{ex12c} and~\eqref{ex12b}. We begin with~\eqref{ex12b} and thus consider $n\ge 2$. Since $\Phi_n$ is a polynomial and $H^1(\Omega,\mathbb{R}^2)$ continuously embeds in $L_\infty(\Omega,\mathbb{R}^2)$, the vector field $D\Phi_n(u_\varepsilon)$ belongs to $H^1(\Omega,\mathbb{R}^2)$. We may then take $v=D\Phi_n(u_\varepsilon)$ in~\eqref{ex12a} to obtain
\begin{equation}
	\int_\Omega \left[ \langle u_{\varepsilon} - U , D\Phi_n(u_\varepsilon) \rangle 
	+ \tau \langle M_{\varepsilon}(u_{\varepsilon}) \partial_x u_{\varepsilon} , \partial_x( D\Phi_n(u_\varepsilon)) \rangle \right]\ \mathrm{d}x = 0 \,. \label{ex13}
\end{equation}
On the one hand, it follows from the convexity of $\Phi_n$ on $[0,\infty)^2$, see Proposition~\ref{prlfa}~(a), that 
\begin{equation}
	\int_\Omega \langle u_{\varepsilon} - U , D\Phi_n(u_\varepsilon) \rangle\ \mathrm{d}x \ge \int_\Omega \left[ \Phi_n(u_\varepsilon) - \Phi_n(U) \right]\ \mathrm{d}x \,. \label{ex14}
\end{equation}
On the other hand, owing to the symmetry of $D^2\Phi_n(u_\varepsilon)$,
\begin{align*}
	\langle M_{\varepsilon}(u_{\varepsilon}) \partial_x u_{\varepsilon} , \partial_x (D\Phi_n(u_\varepsilon)) \rangle & = \langle D^2\Phi_n(u_\varepsilon) M_{\varepsilon}(u_{\varepsilon}) \partial_x u_{\varepsilon} , \partial_x u_{\varepsilon} \rangle \\
	& = \varepsilon \langle D^2\Phi_n(u_\varepsilon) \partial_x u_{\varepsilon} , \partial_x u_{\varepsilon} \rangle \\
	& \qquad + \langle D^2\Phi_n(u_\varepsilon) M(u_{\varepsilon}) \partial_x u_{\varepsilon} , \partial_x u_{\varepsilon} \rangle\,.
\end{align*}
Since both $D^2\Phi_n(u_\varepsilon)$ and $D^2\Phi_n(u_\varepsilon) M(u_{\varepsilon})$ belong to $\mathbf{SPD}_2(\mathbb{R})$ by Proposition~\ref{prlfa}, we conclude that 
\begin{equation}
	\langle M_{\varepsilon}(u_{\varepsilon}) \partial_x u_{\varepsilon} , \partial_x D\Phi_n(u_\varepsilon) \rangle \ge 0\,. \label{ex17}
\end{equation}
Combining \eqref{ex13}, \eqref{ex14}, and \eqref{ex17}, we  end up with
\begin{equation*}
\int_\Omega \left[ \Phi_n(u_{\varepsilon}) - \Phi_n(U) \right]\ \mathrm{d}x  \le 0\,, 
\end{equation*}
and we have established \eqref{ex12b}.
 It next follows from \eqref{ex12b} and Lemma~\ref{lelfb} that
\begin{align*}
\| u_{1,\varepsilon} + u_{2,\varepsilon} \|_n & \le \left( \int_\Omega \Phi_n(u_{\varepsilon})\ \mathrm{d}x \right)^{1/n} \le \left( \int_\Omega \Phi_n(U)\ \mathrm{d}x \right)^{1/n} \\
& \le \frac{1+R}{R} \|U_{1} + U_{2}\|_n\,.
\end{align*}
Hence, letting $n\to\infty$ in the above inequality leads us to \eqref{ex12c}, and the proof is complete.
\end{proof}

We next derive estimates on $\partial_x u_{\varepsilon}$. 

\begin{lemma}\label{lem.ex3}
	Let $\tau>0$, $U\in L_{\infty,+}(\Omega,\mathbb{R}^2)$, and $\varepsilon\in (0,1)$.
	 The weak solution~$u_{\varepsilon} =(u_{1,\varepsilon},u_{2,\varepsilon})$ constructed in Lemma~\ref{lem.ex2} satisfies 
	\begin{equation*}
		\int_\Omega \Phi_1(u_{\varepsilon})\ \mathrm{d}x + \tau \int_\Omega \left[  |\partial_x u_{1,\varepsilon}|^2 + R |\partial_x (u_{1,\varepsilon} + u_{2,\varepsilon})|^2 \right]\ \mathrm{d}x \le \int_\Omega \Phi_1(U)\ \mathrm{d}x\,.
	\end{equation*}
\end{lemma}

\begin{proof}
Let $\eta\in (0,1)$. 
Recalling that $u_{\varepsilon}\in H^1_+(\Omega,\mathbb{R}^2)$ has non-negative components, we deduce that 
 the vector field~$\left( \ln{(u_{1,\varepsilon}+\eta)}, \ln{(u_{2,\varepsilon}+\eta)/\mu} \right)$  belongs to $H^1(\Omega,\mathbb{R}^2)$, and we infer from \eqref{ex12a} that
\begin{align}
	0 & = \int_\Omega \left[ (u_{1,\varepsilon} - U_1) \ln{(u_{1,\varepsilon}+\eta)} + \frac{1}{\mu} (u_{2,\varepsilon} - U_2)\ln{(u_{2,\varepsilon}+\eta)} \right]\ \mathrm{d}x \nonumber \\
	& \qquad + \tau \int_\Omega \Big( m_{\varepsilon,11}(u_{\varepsilon}) \partial_x u_{1,\varepsilon} + m_{\varepsilon,12}(u_{\varepsilon}) \partial_x u_{2,\varepsilon} \Big) \frac{\partial_x u_{1,\varepsilon}}{u_{1,\varepsilon}+\eta}\ \mathrm{d}x \label{ex19}\\
	& \qquad + \frac{\tau}{\mu} \int_\Omega\Big( m_{\varepsilon,21}(u_{\varepsilon}) \partial_x u_{1,\varepsilon} + m_{\varepsilon,22}(u_{\varepsilon}) \partial_x u_{2,\varepsilon} \Big) \frac{\partial_x u_{2,\varepsilon}}{u_{2,\varepsilon}+\eta}\ \mathrm{d}x \,. \nonumber
\end{align}
On the one hand, since $L'(r)=\ln{r}$, $r>0$, the convexity of $L$ guarantees that
\begin{align*}
	& \int_\Omega \left[ (u_{1,\varepsilon} - U_1) \ln{(u_{1,\varepsilon}+\eta)} + \frac{1}{\mu} (u_{2,\varepsilon} - U_2)\ln{(u_{2,\varepsilon}+\eta)} \right]\ \mathrm{d}x \\
	& \qquad \ge \int_\Omega \left[ (L(u_{1,\varepsilon}+\eta) - L(U_1+\eta)) + \frac{1}{\mu} (L(u_{2,\varepsilon}+\eta) - L(U_2+\eta)) \right]\ \mathrm{d}x \\
	& \qquad = \int_\Omega \Phi_1((u_{1,\varepsilon}+\eta,u_{2,\varepsilon}+\eta))\ \mathrm{d}x - \int_\Omega \Phi_1((U_1+\eta,U_2+\eta))\ \mathrm{d}x\,.
\end{align*}
Owing to the continuity of $\Phi_1$ on $[0,\infty)^2$, letting $\eta\to 0$ in the above inequality gives
\begin{equation}
	\begin{split}
	& \liminf_{\eta\to 0} \int_\Omega \left[ (u_{1,\varepsilon} - U_1) \ln{(u_{1,\varepsilon}+\eta)} + \frac{1}{\mu} (u_{2,\varepsilon} - U_2)\ln{(u_{2,\varepsilon}+\eta)} \right]\ \mathrm{d}x \\
	& \qquad \ge \int_\Omega \Phi_1(u_{\varepsilon})\ \mathrm{d}x - \int_\Omega \Phi_1(U)\ \mathrm{d}x\,.
	\end{split}\label{ex20}
\end{equation}
On the other hand,
\begin{equation}\label{ex21}
\begin{aligned}
	D(\eta) & := \tau \int_\Omega \Big( m_{\varepsilon,11}(u_{\varepsilon}) \partial_x u_{1,\varepsilon} + m_{\varepsilon,12}(u_{\varepsilon}) \partial_x u_{2,\varepsilon} \Big) \frac{\partial_x u_{1,\varepsilon}}{u_{1,\varepsilon}+\eta}\ \mathrm{d}x \\
	& \qquad + \frac{\tau}{\mu} \int_\Omega \Big( m_{\varepsilon,21}(u_{\varepsilon}) \partial_x u_{1,\varepsilon} + m_{\varepsilon,22}(u_{\varepsilon}) \partial_x u_{2,\varepsilon} \Big) \frac{\partial_x u_{2,\varepsilon}}{u_{2,\varepsilon}+\eta}\ \mathrm{d}x \\
	& = \tau \varepsilon \int_\Omega \left( \frac{|\partial_x u_{1,\varepsilon}|^2}{u_{1,\varepsilon}+\eta} + \frac{|\partial_x u_{2,\varepsilon}|^2}{u_{2,\varepsilon}+\eta} \right)\ \mathrm{d}x   \\
	& \qquad + \tau \int_\Omega \left[ \frac{u_{1,\varepsilon}}{u_{1,\varepsilon}+\eta} |\partial_x u_{1,\varepsilon}|^2 + R |\partial_x u_{1,\varepsilon} + \partial_x u_{2,\varepsilon}|^2 \right]\ \mathrm{d}x  \\
	& \qquad - \tau R \int_\Omega \left[ \left( 1 - \frac{u_{1,\varepsilon}}{u_{1,\varepsilon}+\eta} \right) \partial_x u_{1,\varepsilon}\cdot \partial_x\left(  u_{1,\varepsilon} +  u_{2,\varepsilon} \right) \right]\ \mathrm{d}x  \\
	& \qquad - \tau R \int_\Omega \left[ \left( 1 - \frac{u_{2,\varepsilon}}{u_{2,\varepsilon}+\eta} \right) \partial_x u_{2,\varepsilon} \cdot\partial_x\left(  u_{1,\varepsilon} +  u_{2,\varepsilon} \right) \right]\ \mathrm{d}x  \\
	& \ge \tau \int_\Omega \left[  |\partial_x u_{1,\varepsilon}|^2 + R |\partial_x (u_{1,\varepsilon} + u_{2,\varepsilon})|^2 \right]\ \mathrm{d}x  \\
	& \qquad - J_0(\eta) - J_1(\eta) - J_2(\eta)\,, 
\end{aligned}
\end{equation}
with
\begin{align*}
	J_0(\eta) & := \tau \int_\Omega \left( 1 - \frac{u_{1,\varepsilon}}{u_{1,\varepsilon}+\eta} \right) |\partial_x u_{1,\varepsilon}|^2\ \mathrm{d}x\,, \\
	J_1(\eta) & := \tau R \int_\Omega \left[ \left( 1 - \frac{u_{1,\varepsilon}}{u_{1,\varepsilon}+\eta} \right) \partial_x u_{1,\varepsilon} \cdot \partial_x \left( u_{1,\varepsilon} + u_{2,\varepsilon} \right) \right]\ \mathrm{d}x\,, \\
	J_2(\eta) & := \tau R \int_\Omega \left[ \left( 1 - \frac{u_{2,\varepsilon}}{u_{2,\varepsilon}+\eta} \right) \partial_x u_{2,\varepsilon} \cdot\partial_x \left( u_{1,\varepsilon} + u_{2,\varepsilon} \right) \right]\ \mathrm{d}x\,.
\end{align*}
Now, $u_{\varepsilon}\in H^1(\Omega,\mathbb{R}^2)$ and, for $j\in\{1,2\}$,  
\begin{align*}
	& \lim_{\eta\to 0} \frac{u_{j,\varepsilon}}{u_{j,\varepsilon}+\eta} = 1 \;\text{ a.e. in }\; \{x\in\Omega\ :\ u_{j,\varepsilon}>0\}\,, \\
	& \lim_{\eta\to 0} \left( 1 - \frac{u_{j,\varepsilon}}{u_{j,\varepsilon}+\eta} \right) \partial_x u_{j,\varepsilon} = 0\;\text{ a.e. in }\; \{x\in\Omega\ :\ u_{j,\varepsilon}=0\}\,, 
\end{align*}
so that we infer from Lebesgue's dominated convergence theorem that
\begin{equation}
	\lim_{\eta\to 0} \left( J_0(\eta) + J_1(\eta) + J_2(\eta) \right) = 0\,. \label{ex22}
\end{equation}
Combining \eqref{ex21} and \eqref{ex22} gives
\begin{equation}
	\liminf_{\eta\to 0} D(\eta) \ge \tau \int_\Omega \left[  |\partial_x u_{1,\varepsilon}|^2 + R |\partial_x (u_{1,\varepsilon} + u_{2,\varepsilon})|^2 \right]\ \mathrm{d}x\,. \label{ex23}
\end{equation}
In view of \eqref{ex20} and \eqref{ex23}, we may pass to the limit $\eta\to 0$ in \eqref{ex19} and obtain the stated inequality.
\end{proof}

\subsection{A time discrete scheme: existence}\label{sec02.2}

Thanks to the analysis performed in the previous section, we are now in a position to take the limit $\varepsilon\to 0$ and prove Proposition~\ref{prop.e1}.

\begin{proof}[Proof of Proposition~\ref{prop.e1}]
Consider $\tau>0$ and $U=(F,G)\in L_{\infty,+}(\Omega,\mathbb{R}^2)$.
 Given  $\varepsilon\in (0,1),$  let~${u_{\varepsilon} =(u_{1,\varepsilon},u_{2,\varepsilon})\in H^1_+(\Omega,\mathbb{R}^2)}$
 denote the weak solution to \eqref{ex12a} provided by Lemma~\ref{lem.ex2}. 
 It first follows from \eqref{ex12c} and the componentwise non-negativity of $u_{\varepsilon}$ that 
\begin{equation}
	\max\left\{ \|u_{1,\varepsilon}\|_\infty, \|u_{2,\varepsilon}\|_\infty \right\}\le \|u_{\varepsilon}\|_\infty \le \frac{1+R}{R} \|F+G\|_\infty\,. \label{cv01}
\end{equation}
In view of the non-negativity of $\Phi_1$, we infer from Lemma~\ref{lem.ex3} that
\begin{equation*}
	\int_\Omega \left[  |\partial_x u_{1,\varepsilon}|^2 + R |\partial_x (u_{1,\varepsilon} + u_{2,\varepsilon})|^2 \right]\ \mathrm{d}x \le \frac{1}{\tau} \int_\Omega \Phi_1(U)\ \mathrm{d}x\,.
\end{equation*}
Hence, by \eqref{ex4},
\begin{equation}
	\|\partial_x u_{\varepsilon}\|_2^2 \le \frac{1+2R}{\tau R} \int_\Omega \Phi_1(U)\ \mathrm{d}x\,. \label{cv03}
\end{equation}
Due to the compactness of the embedding of $H^1(\Omega)$ in $L_\infty(\Omega)$, we deduce from \eqref{cv01} and \eqref{cv03} that there is $u=(f,g)\in H^1_+(\Omega,\mathbb{R}^2)$ and a sequence $(u_{\varepsilon_j})_{j\ge 1}$ such that
\begin{equation}
	\begin{split}
		& u_{\varepsilon_j} \rightharpoonup u \;\text{ in }\; H^1(\Omega,\mathbb{R}^2)\,, \\
		& \lim_{j\to\infty} \|u_{\varepsilon_j} - u\|_\infty = 0\,. 
	\end{split} \label{cv04}
\end{equation}
An immediate consequence of \eqref{ex12c}, \eqref{ex12b}, and \eqref{cv04} is that $(f,g)$ satisfies \eqref{ex2} for $n\ge 2$ and \eqref{ex2a}. 
Moreover, another consequence of \eqref{cv04}, along with Lemma~\ref{lem.ex3} and a weak lower semicontinuity argument, is that $(f,g)$ satisfies \eqref{ex2b}. Finally, owing to \eqref{cv04} and the boundedness of the coefficients of $M_{\varepsilon}(u_{\varepsilon})$ due to \eqref{cv01}, we may use Lebesgue's dominated convergence theorem to take the limit $j\to\infty$ in the identity \eqref{ex12a} for $u_{\varepsilon_j}$ and conclude that $(f,g)$ satisfies \eqref{ex1}, thereby completing the proof of Proposition~\ref{prop.e1}.
\end{proof}

\section{Existence of bounded weak solutions:  $d=1$ }\label{sec03}

This section is devoted to the proof of Theorem~\ref{ThBWS}. 
To this end, we argue in a standard way and construct, starting from the initial condition  $(f^{in},g^{in})\in L_{\infty,+}(\Omega,\mathbb{R}^2)$ and using  Proposition~\ref{prop.e1}, a family of piecewise constant functions $(u^\tau)_{\tau\in(0,1)}$. Specifically, we set $u^\tau(0):=u_0^\tau $ and
\begin{equation}\label{x01}
 u^\tau(t)= u^\tau_{l}\,, \qquad t\in ((l-1)\tau, l\tau]\,, \quad  1\leq l\in\mathbb{N}\,,
\end{equation}
where, given $\tau\in(0,1)$, the sequence $(u_{l}^\tau)_{l\geq0}$ is defined as follows
\begin{equation}
	\begin{split}
	&u_0^\tau := u^{in} =(f^{in},g^{in})\in L_{\infty,+}(\Omega,\mathbb{R}^2)\,, \\
	&u_{l+1}^\tau =(f_{l+1}^\tau,g_{l+1}^\tau)\in H^1_+(\Omega,\mathbb{R}^2) \;\text{is the solution to \eqref{ex1}} \\
	&\text{with $ U=u_l^\tau=(f_l^\tau,g_l^\tau)$ constructed in Proposition~\ref{prop.e1} for $  l\ge 0$}\,.
	\end{split}\label{x02}
\end{equation}
In order to establish Theorem~\ref{ThBWS}, we show that the family $(u^\tau)_{\tau\in(0,1)}$  converges along a subsequence~${\tau_j\to0}$ towards a pair $u=(f,g)$ which fulfills all the requirements of Theorem~\ref{ThBWS}. 
\medskip

Throughout this section, $C$ and $C_i$, with~${i\ge 0}$, denote positive constants depending only on~$R$,~$\mu$, and~${(f^{in},g^{in})}$. 
Dependence upon additional parameters will be indicated explicitly.

\begin{proof}[Proof of Theorem~\ref{MainTh}]
Let $\tau\in (0,1)$ and let $u^\tau$ be defined in \eqref{x01}-\eqref{x02}.  Given   $l\ge 0$, we infer from Proposition \ref{prop.e1} that  
\begin{subequations}\label{x03}
	\begin{align}
		\int_\Omega \Big[ f_{l+1}^\tau \varphi + \tau f_{l+1}^\tau \partial_x\left[ (1+R) f_{l+1}^\tau + R g_{l+1}^\tau \right] \partial_x\varphi \Big]\ \mathrm{d}x & = \int_\Omega f_{l}^\tau \varphi\ \mathrm{d}x\,, \qquad \varphi\in H^1(\Omega)\,, \label{x03a} \\
		\int_\Omega \Big[ g_{l+1}^\tau \psi + \tau \mu R g_{l+1}^\tau \partial_x(f_{l+1}^\tau + g_{l+1}^\tau) \partial_x\psi \Big]\ \mathrm{d}x & = \int_\Omega g_{l}^\tau \psi\ \mathrm{d}x\,, \qquad \psi\in H^1(\Omega)\,. \label{x03b}
	\end{align}
\end{subequations}
Moreover,
\begin{equation}
	\int_\Omega \Phi_n(u_{l+1}^\tau)\ \mathrm{d}x \le \int_\Omega \Phi_n(u_{l}^\tau)\ \mathrm{d}x \label{x04}
\end{equation}
for $n\ge 2$ and we also have
\begin{equation}
	\int_\Omega \Phi_1(u_{l+1}^\tau)\ \mathrm{d}x + \tau \int_\Omega \left[  |\partial_x f_{l+1}^\tau|^2 + R |\partial_x (f_{l+1}^\tau + g_{l+1}^\tau)|^2 \right]\ \mathrm{d}x \le \int_\Omega \Phi_1(u_{l}^\tau)\ \mathrm{d}x\,. \label{x05}
\end{equation}
 It readily follows from \eqref{x01}, \eqref{x02}, \eqref{x04}, and \eqref{x05} that, for $t>0$,
\begin{equation}
	\int_\Omega \Phi_1(u^\tau(t))\ \mathrm{d}x + \int_0^t \int_\Omega \left[  |\partial_x f^\tau(s)|^2 + R |\partial_x (f^\tau + g^\tau)(s)|^2 \right]\ \mathrm{d}x\mathrm{d}s \le \int_\Omega \Phi_1(u^{in})\ \mathrm{d}x\,, \label{x07}
\end{equation}
and
\begin{equation}
	\int_\Omega \Phi_n(u^\tau(t))\ \mathrm{d}x \le \int_\Omega \Phi_n(u^{in})\ \mathrm{d}x\,, \qquad \ n\ge 2\,. \label{x06}
\end{equation}

An immediate consequence of \eqref{x06} and Lemma~\ref{lelfb} is the estimate 
\begin{equation*}
	\|f^\tau(t)+g^\tau(t)\|_n \le \frac{1+R}{R} \|f^{in}+g^{in}\|_n\,, \qquad n\ge 2\,, \ t>0\,.
\end{equation*}
Letting $n\to\infty$ in the above inequality gives
\begin{equation}
	\|f^\tau(t)+g^\tau(t)\|_\infty \le C_1 := \frac{1+R}{R} \|f^{in}+g^{in}\|_\infty\,, \qquad t>0\,. \label{x08}
\end{equation}
Also, it readily follows from \eqref{ex4}, \eqref{x07}, and the non-negativity of $\Phi_1$ that
\begin{align*}
	& \frac{R}{1+2R} \int_0^t \int_\Omega \left(  |\partial_x f^\tau(s)|^2 + |\partial_x g^\tau(s)|^2 \right)\ \mathrm{d}x\mathrm{d}s \\
	& \qquad \le \int_\Omega \Phi_1(u^\tau(t))\ \mathrm{d}x + \int_0^t \int_\Omega \left[  |\partial_x f^\tau(s)|^2 + R |\partial_x (f^\tau + g^\tau)(s)|^2 \right]\ \mathrm{d}x\mathrm{d}s \\
	& \qquad \le \int_\Omega \Phi_1(u^{in})\ \mathrm{d}x\,.
\end{align*}
Therefore we have
\begin{equation}
	\int_0^t \left( \|\partial_x f^\tau(s)\|_2^2 + \|\partial_x g^\tau(s)\|_2^2 \right)\ \mathrm{d}s \le C_2 := \frac{1+2R}{R} \int_\Omega \Phi_1(u^{in})\ \mathrm{d}x\,, \qquad t>0\,. \label{x09}
\end{equation}

Next, for $l\ge 1$ and $t\in ((l-1)\tau,l\tau]$, we deduce from \eqref{x03a}, \eqref{x08}, and H\"older's inequality that, for $\varphi\in H^1(\Omega)$, 
\begin{align*}
	\left| \int_\Omega \left( f^{\tau}(t+\tau) - f^\tau(t) \right) \varphi\ \mathrm{d}x \right| & = \left| \int_{l\tau}^{(l+1) \tau} \int_\Omega f^\tau(s) \partial_x\left[ (1+R) f^\tau(s) + R g^\tau(s) \right] \partial_x\varphi\ \mathrm{d}x\mathrm{d}s \right| \\
	& \le \int_{l\tau}^{(l+1) \tau} \| f^\tau(s)\|_\infty \|\partial_x\left[(1+R) f^\tau(s) + R g^\tau(s)\right]\|_2 \|\partial_x\varphi\|_2\ \mathrm{d}s \\
	& \le C_1 \|\partial_x\varphi\|_2 \int_{l\tau}^{(l+1) \tau} \|\partial_x\left[ (1+R) f^\tau(s) + R g^\tau(s) \right]\|_2\ \mathrm{d}s \,.
\end{align*}
A duality argument then gives
\begin{equation*}
	\| f^{\tau}(t+\tau) - f^\tau(t) \|_{(H^1)'} \le  C_1 \int_{l\tau}^{(l+1) \tau} \|\partial_x\left[ (1+R) f^\tau(s) + R g^\tau(s) \right]\|_2\ \mathrm{d}s\,, \quad t\in ((l-1)\tau,l\tau]\,,l\geq1\,.
\end{equation*}
Now, for $L\ge 2$ and $T\in ((L-1)\tau,L\tau]$, the above inequality, along with H\"older's inequality, entails that
\begin{align*}
	\int_0^{T-\tau} \| f^{\tau}(t+\tau) - f^\tau(t) \|_{(H^1)'}^2\ \mathrm{d}t & \le \int_0^{(L-1)\tau} \| f^{\tau}(t+\tau) - f^\tau(t) \|_{(H^1)'}^2\ \mathrm{d}t \\
	&  =\sum_{l=1}^{L-1} \int_{(l-1)\tau}^{l\tau} \| f^{\tau}(t+\tau) - f^\tau(t) \|_{(H^1)'}^2\ \mathrm{d}t \\
	& \le C_1^2 \tau \sum_{l=1}^{L-1} \left( \int_{l\tau}^{(l+1) \tau} \|\partial_x\left[ (1+R) f^\tau(s) + R g^\tau(s) \right]\|_2\ \mathrm{d}s \right)^2\\
	& \le C_1^2 \tau^2 \sum_{l=1}^{L-1} \int_{l\tau}^{(l+1) \tau} \|\partial_x\left[ (1+R) f^\tau(s) + R g^\tau(s) \right]\|_2^2\ \mathrm{d}s \\
	& \le C_1^2 \tau^2 \int_{0}^{L \tau} \|\partial_x\left[ (1+R) f^\tau(s) + R g^\tau(s) \right]\|_2^2\ \mathrm{d}s\,.
\end{align*}
We then use \eqref{x09} (with $t=L\tau$) and Young's inequality to obtain
\begin{align}
	\int_0^{T-\tau} \| f^{\tau}(t+\tau) - f^\tau(t) \|_{(H^1)'}^2\ \mathrm{d}t & \le C_1^2 \tau^2 \int_{0}^{L\tau} \left( 2 (1+R)^2 \|\nabla f^\tau(s)\|_2^2 + 2R^2 \|\nabla g^\tau(s)\|_2^2 \right)\ \mathrm{d}s \nonumber \\
	& \le C_3 \tau^2 \,, \label{x10}
\end{align}
with $C_3 := 2 (1+R)^2 C_1^2 C_2$. Similarly, 
\begin{equation}
	\int_0^{T-\tau} \| g^{\tau}(t+\tau) - g^\tau(t) \|_{(H^1)'}^2\ \mathrm{d}t \le C_4 \tau^2 \,, \label{x11}
\end{equation}
with $C_4 := 2 \mu^2 R^2 C_1^2 C_2$. 

Since $H^1(\Omega,\mathbb{R}^2)$ is compactly embedded in $L_2(\Omega,\mathbb{R}^2)$ and $L_2(\Omega,\mathbb{R}^2)$ is continuously embedded in $H^1(\Omega,\mathbb{R}^2)'$, we infer from \eqref{x08}, \eqref{x09}, \eqref{x10}, \eqref{x11}, and \cite[Theorem~1]{DJ2012} that, for any $T>0$,
\begin{equation}
	(u^\tau)_{\tau\in (0,1)} \;\text{ is relatively compact in }\; L_2((0,T)\times\Omega,\mathbb{R}^2)\,. \label{x12}
\end{equation}
Owing to \eqref{x08}, \eqref{x09}, and \eqref{x12}, we may use a Cantor diagonal argument to find a
 function~${u:=(f,g)}$ in~$L_{\infty,+}((0,\infty)\times\Omega,\mathbb{R}^2)$ and a sequence $(\tau_j)_{j\ge 1}$, $\tau_j\to 0$, such that, for any $T>0$ and $p\in [1,\infty)$,
\begin{equation}
	\begin{split}
		u^{\tau_j} & \longrightarrow u \;\;\text{ in }\;\; L_p((0,T)\times\Omega,\mathbb{R}^2)\,, \\
		u^{\tau_j} & \stackrel{*}{\rightharpoonup} u \;\;\text{ in }\;\; L_\infty((0,T)\times\Omega,\mathbb{R}^2)\,, \\
		u^{\tau_j} & \rightharpoonup u \;\;\text{ in }\;\; L_2((0,T),H^1(\Omega,\mathbb{R}^2))\,.
	\end{split} \label{x13}
\end{equation}
In addition, the compact embedding of $L_2(\Omega,\mathbb{R}^2)$ in $H^1(\Omega,\mathbb{R}^2)'$, along with \eqref{x06} with $n=2$, \eqref{x10}, and \eqref{x11}, allows us to apply once more \cite[Theorem~1]{DJ2012} to conclude that 
\begin{equation}
	u\in C([0,\infty),H^1(\Omega,\mathbb{R}^2)')\,. \label{x14}
\end{equation}

Let us now identify the equations solved by the components $f$ and $g$ of $u$. 
To this end, let~${\chi\in W^1_\infty([0,\infty))}$ be a compactly supported function and $\varphi\in C^1(\bar{\Omega})$. In view of  \eqref{x03a}, classical computations give 
\begin{align*}
	& \int_0^\infty \int_\Omega \frac{\chi(t+\tau)-\chi(t)}{\tau} f^\tau(t) \varphi\ \mathrm{d}x\mathrm{d}t + \left( \frac{1}{\tau} \int_0^\tau \chi(t)\ \mathrm{d}t \right) \int_\Omega f^{in}\varphi\ \mathrm{d}x 	\\
	& \qquad = \int_0^\infty \int_\Omega \chi(t) f^\tau(t) \partial_x\left[ (1+R)f^\tau(t) + R g^\tau(t)\right] \partial_x\varphi\ \mathrm{d}x\mathrm{d}t\,.
\end{align*}
Taking $\tau=\tau_j$ in the above identity, it readily follows from \eqref{x13} and the regularity of $\chi$ and $\varphi$ that we may pass to the limit as $j\to\infty$ and conclude that
\begin{equation}
\begin{split}
	& \int_0^\infty \int_\Omega \frac{d\chi}{dt}(t) f(t) \varphi\ \mathrm{d}x\mathrm{d}t + \chi(0) \int_\Omega f^{in}\varphi\ \mathrm{d}x 	\\
	& \qquad = \int_0^\infty \int_\Omega \chi(t) f(t) \partial_x\left[ (1+R)f(t) + R g(t)\right] \partial_x \varphi\ \mathrm{d}x\mathrm{d}t\,.
\end{split}\label{x15}
\end{equation}
Since $f\partial_x f$ and $f\partial_x g$ belong to $L_2((0,T)\times\Omega)$ for all $T>0$ by \eqref{x13}, a density argument ensures that the identity~\eqref{x15} is valid for any $\varphi\in H^1(\Omega)$. We next use the time continuity \eqref{x14} of $f$ and a classical approximation argument to show that $f$ solves \eqref{p2a}. A similar argument allows us to derive \eqref{p2b} from \eqref{x03b}.

Finally, combining \eqref{x13}, \eqref{x14}, and a weak lower semicontinuity argument, we may let $j\to\infty$ in \eqref{x07}, \eqref{x06}, and \eqref{x08} with $\tau=\tau_j$ to show that $u=(f,g)$ satisfies \eqref{p3}, \eqref{p4a}, and \eqref{p5}, thereby completing the proof.
\end{proof}

We end up this section with the proof of Corollary~\ref{Cor:1}.

\begin{proof}[Proof of Corollary~\ref{Cor:1}]
Assume that  $R\max\{1,\mu\}\in(0,1/(2e)]$. Given an integer $m\ge 1$, we define the function  $\xi:(0,1/(2e)]\to\mathbb{R}$  by the formula 
\begin{equation*}
\xi(y):=\exp\left\{m \left[ (1+ y) \ln\Big(1+\frac{1}{y}\Big) - 1 \right]\right\}-1.
\end{equation*}
It then holds
\begin{align*}
y^m\xi(y)&=(1+y)^m\exp\left\{m \left[  y  \ln\Big(1+\frac{1}{y}\Big) - 1 \right]\right\}-y^m>\frac{(1+y)^m}{e^m} -y^m\geq \frac{1}{e^m} -\frac{1}{(2e)^m}\geq\frac{1}{2e^m}\,.
\end{align*}
Consequently, the constant $\nu_n$ defined in Lemma~\ref{lelfb} satisfies
\begin{equation*}
\nu_n>\frac{1}{2(eR\max\{1,\mu\})^{n-1}},\qquad n\geq 2,
\end{equation*}
We then infer from Theorem~\ref{ThBWS}~(p3), the above inequality, and~\eqref{LUB} that, for $t>0$ and $n\ge 2$, 
\begin{align*}
	\|f(t)\|_n^n & \le \frac{1}{\nu_n} \int_\Omega \Phi_n((f(t),g(t)))\ \mathrm{d}x \le \frac{1}{\nu_n} \int_\Omega \Phi_n((f^{in},g^{in}))\ \mathrm{d}x \\
	& \le \frac{2(eR\max\{1,\mu\})^{n-1}}{R^n} \|(1+R) f^{in} + R g^{in}\|_n^n\,.
\end{align*}
Hence,
\begin{equation*}
	\|f(t)\|_n \le \left( \frac{2}{R} \right)^{1/n} (e\max\{1,\mu\})^{(n-1)/n} \|(1+R) f^{in} + R g^{in}\|_n\,.
\end{equation*}
Taking the limit $n\to\infty$ in the above inequality gives 
\begin{equation*}
	\|f(t)\|_\infty \le  e\max\{1,\mu\} \|(1+R) f^{in} + R g^{in}\|_\infty\,,
\end{equation*}
and we use the upper bound $e R\max\{1,\mu\}\le 1$ to obtain the desired estimate~\eqref{estaaa}.
\end{proof}

\section{Proof of Theorem~\ref{MainTh}}\label{sec04}

\begin{proof}[Proof of Theorem~\ref{MainTh}]	
Owing to Proposition~\ref{prlfa}, the proof of Theorem~\ref{MainTh} is a simple application of the scheme described in \eqref{in4} and \eqref{in5}. 
Indeed, let $u=(f,g)$ be a sufficiently regular solution to \eqref{tfm1} on $[0,\infty)$ and $n\ge 2$.
 It follows from the alternative form~\eqref{in4} of the system~\eqref{tfm1a}-\eqref{tfm1b} and the boundary conditions~\eqref{tfm1c} that
\begin{equation*}
	\frac{d}{dt} \int_\Omega \Phi_n(u)\ \mathrm{d}x + \sum_{i=1}^d \int_\Omega \langle D^2\Phi_n(u) M(u) \partial_i u , \partial_i u \rangle\ \mathrm{d}x = 0\,. 
\end{equation*} 
According to \eqref{lf03} and Proposition~\ref{prlfa}~(b), we infer from the componentwise non-negativity of $u$  and the continuity of $D^2\Phi_n M$ that 
\begin{equation*}
	\langle D^2\Phi_n(u) M(u) \partial_i u , \partial_i u \rangle \ge 0 \;\;\text{ in }\;\; (0,\infty)\times\Omega\,, \qquad 1\le i \le d\,. 
\end{equation*}
Consequently, 
\begin{equation*}
	\frac{d}{dt} \int_\Omega \Phi_n(u)\ \mathrm{d}x \le 0\,, \qquad t>0\,,
\end{equation*}
and \eqref{p4a} is proved. In particular, thanks to \eqref{LUB}, we have shown that, for $t>0$ and $n\ge 2$,
\begin{equation*}
	\|f(t)+g(t)\|_n \le \left( \int_\Omega \Phi_n(u(t))\ \mathrm{d}x \right)^{1/n} \le \left( \int_\Omega \Phi_n(u^{in})\ \mathrm{d}x \right)^{1/n} \le \frac{1+R}{R} \| f^{in} + g^{in} \|_n\,.
\end{equation*}
Taking the limit $n\to\infty$ in the above inequality gives \eqref{p5}. Finally, to establish the inequality~\eqref{p3}, we use \eqref{tfm1} to compute the time derivative of 
\begin{equation*}
	\int_\Omega \Phi_1((f(t)+\eta,g(t)+\eta))\ \mathrm{d}x
\end{equation*}
and argue as in the proof of Lemma~\ref{lem.ex3} to derive~\eqref{p3}.
\end{proof}

\section*{Acknowledgments}

PhL gratefully acknowledges the hospitality and support of the Fakult\"at f\"ur Mathematik, Universit\"at Regensburg, where this work was done.

\appendix
%
\section{Properties of the polynomials $\Phi_n$, $n\ge 2$}\label{secapA}

In this section, we establish some important properties of the polynomials $\Phi_n$, $n\ge 2$, defined in~\eqref{p4b} {and~\eqref{p4c}}, which lead to Theorem~\ref{MainTh}
 according to the scheme outlined in the Introduction, see \eqref{in4}-\eqref{in5}, and are extensively used in Section~\ref{sec02}, see the proof of Lemma~\ref{lem.ex2}.
  Let thus $n\geq 2$. To begin with, we  recall that $a_{0,n}=1$,
\begin{subequations}\label{lf01}
	\begin{equation}
		a_{j,n} = \binom{n}{j} \prod_{k=0}^{j-1} \frac{k +\alpha_{k,n}}{\alpha_{k,n}}\,, \qquad  1\le j \le n\,, \label{lf01a}
	\end{equation}
	where
	\begin{equation}
		\alpha_{k,n} = R [ k + \mu(n-k-1)] = \mu R(n-1) + R(1-\mu)k > 0\,, \qquad 0 \le k \le n-1\,, \label{lf01b}
	\end{equation}
\end{subequations}
and
\begin{equation}
	\Phi_n(X) := \sum_{j=0}^n a_{j,n} X_1^j X_2^{n-j}\,, \qquad X=(X_1,X_2)\in \mathbb{R}^2\,. \label{lf02} 
\end{equation}
Also, the mobility matrix $M\in C^\infty(\mathbb{R}^2,\mathbf{M}_2(\mathbb{R}))$ is defined in \eqref{moma} by
\begin{equation*}
	M(X) := \begin{pmatrix}
		(1+R) X_1 & R X_1 \\
		\mu R X_2 & \mu R X_2
	\end{pmatrix} \,, \qquad X\in \mathbb{R}^2\,.
\end{equation*}
The aim of this section is twofold. On the one hand, we establish the convexity of $\Phi_n$ 
on $[0,\infty)^2$ and actually show that its Hessian matrix $D^2\Phi_n\in C^\infty(\mathbb{R}^2,\mathbf{Sym}_2(\mathbb{R}))$, defined as usual by 
\begin{equation*}
	D^2\Phi_n(X) = \begin{pmatrix}
		\partial_1^2 \Phi_n(X) & \partial_1 \partial_2 \Phi_n(X)\\
		\partial_1 \partial_2 \Phi_n(X) & \partial_2^2 \Phi_n(X)
	\end{pmatrix} \,, \qquad X\in \mathbb{R}^2\,,
\end{equation*}
is positive definite on $[0,\infty)^2\setminus\{(0,0)\}$. On the other hand, we prove that the matrix
\begin{equation}
	S_n(X) := D^2\Phi_n(X) M(X)\,, \qquad X\in \mathbb{R}^2\,, \label{lf03}
\end{equation}
belongs to  $\mathbf{Sym}_2(\mathbb{R})$ and actually lies in $\mathbf{SPD}_2(\mathbb{R})$ for $X\in (0,\infty)^2$.

\pagebreak
\begin{proposition}\label{prlfa}
	Let $n\ge 2$. 
	\begin{itemize}
		\item [(a)] The polynomial $\Phi_n$ is non-negative and convex on $[0,\infty)^2$. Moreover, we have:
		\begin{itemize}
		\item[(a1)] The gradient $D\Phi_n(X)$ belongs to $[0,\infty)^2$ provided that $X\in[0,\infty)^2;$
		\item[(a2)] The Hessian matrix $D^2\Phi_n(X)$ belongs to  $\mathbf{SPD}_2(\mathbb{R})$ for all $X\in [0,\infty)^2\setminus \{(0,0)\}$.
		\end{itemize}
		\item[(b)] Given $X\in \mathbb{R}^2$, the matrix $S_n(X)$ defined in \eqref{lf03} is symmetric. 
		In addition, it holds that~${S_n(X)\in \mathbf{SPD}_2(\mathbb{R})}$ for all $X\in (0,\infty)^2$.
	\end{itemize}
\end{proposition}

\begin{proof} The proof in the case $n=2$ is a simple exercise. Let now $n\geq3$. 
 We first note that \eqref{lf01} implies that $(a_{j,n})_{1\le j \le n}$ satisfies the following recursion formula
	\begin{equation}
		a_{j+1,n} = \frac{(n-j)(j + \alpha_{j,n})}{(j+1)\alpha_{j,n}} a_{j,n}\,, \qquad 0 \le j \le n-1\,, \label{lf04}
	\end{equation}
	from which we deduce that 
	\begin{equation}
		a_{j,n} > 0\,, \qquad 0\le j \le n\,. \label{lf05}
	\end{equation}
	In particular, $\Phi_n$ is non-negative on $[0,\infty)^2$ and, since
	\begin{align*}
		\partial_1 \Phi_n(X) & = \sum_{j=0}^{n-1} (j+1) a_{j+1,n} X_1^j X_2^{n-j-1}\,, \qquad X\in\mathbb{R}^2\,,\\
		\partial_2 \Phi_n(X) & = \sum_{j=0}^{n-1} (n-j) a_{j,n} X_1^j X_2^{n-j-1}\,, \qquad X\in\mathbb{R}^2\,,
	\end{align*}
the gradient $D\Phi_n(X)$ belongs to $[0,\infty)^2$ for $X\in [0,\infty)^2$,  which proves~(a1).
	
	\medskip
	
	\noindent\textsl{Convexity of $\Phi_n$ on $[0,\infty)^2$}. 
	The convexity of $\Phi_n$ on $[0,\infty)^2$ is a consequence of the property~(a2) which we establish now.	Let $X\in [0,\infty)^2$. 
	We then have 
	\begin{align*} 
		\partial_1^2 \Phi_n(X) & = \sum_{j=1}^{n-1} j(j+1) a_{j+1,n} X_1^{j-1} X_2^{n-j-1} = \sum_{j=0}^{n-2} (j+1)(j+2) a_{j+2,n} X_1^j X_2^{n-j-2} \,, \\
		\partial_1 \partial_2 \Phi_n(X) & = \sum_{j=1}^{n-1} j(n-j) a_{j,n} X_1^{j-1} X_2^{n-j-1} = \sum_{j=0}^{n-2} (j+1)(n-j-1) a_{j+1,n} X_1^j X_2^{n-j-2} \,, \\
		\partial_2^2 \Phi_n(X) &  = \sum_{j=0}^{n-2} (n-j)(n-j-1) a_{j,n} X_1^j X_2^{n-j-2} \,. 
	\end{align*}
	It then readily follows from \eqref{lf05} that the Hessian matrix $D^2\Phi_n(X)$  has a non-negative trace
	\begin{equation}
		\mathrm{tr}(D^2\Phi_n(X)) := \partial_1^2 \Phi_n(X) + \partial_2^2 \Phi_n(X) \ge 0\,, \qquad X\in [0,\infty)^2\,. \label{lf06}
	\end{equation}
	Next,
	\begin{align}
		\det(D^2\Phi_n(X)) & =  \partial_1^2 \Phi_n(X) \partial_2^2 \Phi_n(X) - [\partial_1 \partial_2 \Phi_n(X)]^2 \nonumber \\
		& = \sum_{j=0}^{n-2} \sum_{k=0}^{n-2} (j+1)(n-k-1) A_{j,k} X_1^{j+k} X_2^{2n-j-k-4}\,, \label{lf07}
	\end{align}
	where
	\begin{equation*}
		A_{j,k} := (j+2)(n-k) a_{j+2,n} a_{k,n} - (n-j-1)(k+1) a_{j+1,n} a_{k+1,n}\,, \qquad 0 \le j,k \le n-2\,.
	\end{equation*}
	We now simplify the above formula for $A_{j,k}$ and first use \eqref{lf04} to replace $a_{j+2,n}$ and $a_{k+1,n}$ and subsequently  the definition \eqref{lf01b} of $\alpha_{k,n}$, thereby obtaining 
	\begin{align}
		A_{j,k} & = (n-j-1)(n-k) \frac{j+1+\alpha_{j+1,n}}{\alpha_{j+1,n}} a_{j+1,n} a_{k,n} - (n-j-1)(n-k) \frac{k+\alpha_{k,n}}{\alpha_{k,n}} a_{j+1,n} a_{k,n} \nonumber \\
		& = \mu R (n-1) \frac{(n-j-1)(n-k) ( j+1-k )}{\alpha_{j+1,n} \alpha_{k,n}} a_{j+1,n} a_{k,n} \label{lf08}
	\end{align}
	for $0 \le j,k \le n-2$. In particular, 
	\begin{equation}
		A_{k-1,j+1} = - A_{j,k}\,, \qquad 0\le j\le n-3\,, \ 1\le k\le n-2\,. \label{lf09}
	\end{equation} 
	It then follows from \eqref{lf07} and \eqref{lf08} that
	\begin{align*}
		2 \det(D^2\Phi_n(X))	& = \sum_{j=0}^{n-2} \sum_{k=0}^{n-2} (j+1)(n-k-1) A_{j,k} X_1^{j+k} X_2^{2n-j-k-4} \\
		& \qquad + \sum_{l=1}^{n-1} \sum_{i=-1}^{n-3} l(n-i-2) A_{l-1,i+1} X_1^{i+l} X_2^{2n-i-l-4} \\
		& = \sum_{j=0}^{n-2} \sum_{k=0}^{n-2} (j+1)(n-k-1) A_{j,k} X_1^{j+k} X_2^{2n-j-k-4} \\
		& \qquad + \sum_{j=-1}^{n-3} \sum_{k=1}^{n-1} k(n-j-2) A_{k-1,j+1} X_1^{j+k} X_2^{2n-j-k-4}\\
		& = \sum_{j=0}^{n-3} \sum_{k=1}^{n-2} (j+1)(n-k-1) A_{j,k} X_1^{j+k} X_2^{2n-j-k-4} \\
		& \qquad + \sum_{k=0}^{n-2} (n-1)(n-k-1) A_{n-2,k} X_1^{n-2+k} X_2^{n-k-2} \\
		& \qquad + \sum_{j=0}^{n-3} (j+1)(n-1) A_{j,0} X_1^{j} X_2^{2n-j-4} \\	
		& \qquad  + \sum_{j=0}^{n-3} \sum_{k=1}^{n-2} k(n-j-2) A_{k-1,j+1} X_1^{j+k} X_2^{2n-j-k-4} \\
		& \qquad + \sum_{k=1}^{n-1}  k(n-1) A_{k-1,0} X_1^{k-1} X_2^{2n-k-3} \\
		& \qquad + \sum_{j=0}^{n-3} (n-1)(n-j-2) A_{n-2,j+1} X_1^{j+n-1} X_2^{n-j-3} \,.
	\end{align*}
	Owing to \eqref{lf05} and \eqref{lf08}, $A_{l,0}> 0$ and $A_{n-2,l}>0$ for $0\le l\le n-2$, 
	so that the terms in the above identity involving a single sum are non-negative. 
	Therefore, using the symmetry property~\eqref{lf09} and retaining in the last two sums only the terms corresponding to $k=1$ and $j=n-3$, respectively, we get
	\begin{align*}
		2 \det(D^2\Phi_n(X))	& \ge	 \sum_{j=0}^{n-3} \sum_{k=1}^{n-2} \left[ (j+1)(n-k-1) - k(n-j-2) \right] A_{j,k} X_1^{j+k} X_2^{2n-j-k-4} \\
		& \qquad + (n-1) A_{n-2,n-2} X_1^{2n-4} + (n-1) A_{0,0} X_2^{2n-4} \\
		& = \sum_{j=0}^{n-3} \sum_{k=1}^{n-2} (n-1)(j+1-k) A_{j,k} X_1^{j+k} X_2^{2n-j-k-4} \\
		& \qquad + (n-1) A_{n-2,n-2} X_1^{2n-4} + (n-1) A_{0,0} X_2^{2n-4}\,.
	\end{align*}
	Since
	\begin{equation*}
		(n-1)(j+1-k) A_{j,k} = \mu R (n-1)^2 \frac{(n-j-1)(n-k) ( j+1-k )^2}{\alpha_{j+1,n} \alpha_{k,n}} a_{j+1,n} a_{k,n}  \ge 0
	\end{equation*}
	for $0\le j, k\le n-2$ by \eqref{lf01b}, \eqref{lf05}, and \eqref{lf08}, we conclude that
	\begin{equation}
		\det(D^2\Phi_n(X)) \ge (n-1) A_{n-2,n-2} X_1^{2n-4} + (n-1) A_{0,0} X_2^{2n-4}\,, \qquad X\in [0,\infty)^2\,. \label{lf10}
	\end{equation} 
	Since $A_{0,0}>0$ and $A_{n-2,n-2}>0$, we have thus established that, for each $X\in [0,\infty)^2\setminus\{(0,0)\}$, 
	the symmetric matrix $D^2\Phi_n(X)$ has non-negative trace and positive determinant, so that it is positive definite. This proves~(a2).
	
	\medskip
	
	\noindent\textsl{Symmetry of $S_n(X)$}. Let $X\in \mathbb{R}^2$. It follows from \eqref{lf03} that
		\begin{align*}
				[S_n(X)]_{11} & = (1+R) X_1 \partial_1^2 \Phi_n(X) + \mu R X_2 \partial_1 \partial_2 \Phi_n(X)\\ 
				& = (1+R) \sum_{j=1}^{n-1} j(j+1) a_{j+1,n} X_1^j X_2^{n-j-1} + \mu R \sum_{j=0}^{n-2} (j+1)(n-j-1) a_{j+1,n} X_1^j X_2^{n-j-1}\,, \\
				[S_n(X)]_{12} & = RX_1 \partial_1^2\Phi_n(X) + \mu R X_2 \partial_1 \partial_2 \Phi_n(X) \\
				& = R \sum_{j=1}^{n-1} j(j+1) a_{j+1,n} X_1^j X_2^{n-j-1} + \mu R \sum_{j=0}^{n-2} (j+1)(n-j-1) a_{j+1,n} X_1^j X_2^{n-j-1}\,,\\
				[S_n(X)]_{21} & = (1+ R)X_1 \partial_1 \partial_2 \Phi_n(X) + \mu R X_2 \partial_2^2 \Phi_n(X)\\
				& = (1+R) \sum_{j=1}^{n-1} j(n-j) a_{j,n} X_1^{j} X_2^{n-j-1}  + \mu R \sum_{j=0}^{n-2} (n-j)(n-j-1) a_{j,n} X_1^j X_2^{n-j-1}\,,\\ 
				[S_n(X)]_{22} & = R X_1 \partial_1 \partial_2 \Phi_n(X) + \mu R X_2 \partial_2^2 \Phi_n(X)\\
				& = R \sum_{j=1}^{n-1} j(n-j) a_{j,n} X_1^{j} X_2^{n-j-1} + \mu R \sum_{j=0}^{n-2} (n-j)(n-j-1) a_{j,n} X_1^j X_2^{n-j-1}\,.
		\end{align*}
	It then holds
	\begin{equation*}
		[S_n(X)]_{12} = R n(n-1) a_{n,n} X_1^{n-1} + \sum_{j=1}^{n-2} (j+1) \alpha_{j,n} a_{j+1,n} X_1^j X_2^{n-j-1} + \mu R (n-1) a_{1,n} X_2^{n-1}\,.
	\end{equation*}
	Using the recursion formula \eqref{lf04} and the definition  \eqref{lf01b} of $\alpha_{j,n}$, we get  
	\begin{align*}
		[S_n(X)]_{12} & = R(n-1) \frac{n-1+\alpha_{n-1,n}}{\alpha_{n-1,n}} a_{n-1,n} X_1^{n-1} + \sum_{j=1}^{n-2} (n-j)(j+\alpha_{j,n}) a_{j,n} X_1^j X_2^{n-j-1} \\
		& \qquad + \mu R n(n-1) a_{0,n} X_2^{n-1} \\
		& = (1+R) (n-1) a_{n-1,n} X_1^{n-1} + (1+R) \sum_{j=1}^{n-2} j(n-j) a_{j,n} X_1^j X_2^{n-j-1} \\
		& \qquad + \mu R \sum_{j=1}^{n-2} (n-j)(n-j-1) a_{j,n} X_1^j X_2^{n-j-1} + \mu R n(n-1) a_{0,n} X_2^{n-1} \\
		& = [S_n(X)]_{21}\,,
	\end{align*}
	so that $S_n(X)\in \mathbf{Sym}_2(\mathbb{R})$. 
	
	\medskip
	
	\noindent\textsl{Positive definiteness of $S_n(X)$}. Let $X\in [0,\infty)^2$. It readily follows from \eqref{lf05} that  $[S_n(X)]_{11}\geq0$ and~$[S_n(X)]_{22}\geq0,$ hence
	\begin{equation}
		\mathrm{tr}(S_n(X)) \ge 0\,. \label{lf12}
	\end{equation}
	Moreover, \eqref{lf03} and \eqref{lf10} imply that
	\begin{equation}
		\det(S_n(X)) = \det(D^2\Phi_n(X)) \det(M(X)) = \mu R X_1 X_2 \det(D^2\Phi_n(X)) \ge 0\,. \label{lf13}
	\end{equation}
	Consequently, $S_n(X)$ is a positive semidefinite symmetric matrix for each $X\in [0,\infty)^2$. Moreover, if~${X\in (0,\infty)^2}$, then $\det(S_n(X)) >0$ 
	by \eqref{lf10} and \eqref{lf13}, so that $S_n(X)\in \mathbf{SPD}_2(\mathbb{R})$. This completes the proof of (b). 
\end{proof}

We next derive lower and upper bounds for $\Phi_n$, $n\geq 2$.

\begin{lemma}\label{lelfb}
	Given $n\ge 2$, we have
	\begin{equation}\label{LUB}
		\nu_n X_1^n + (X_1+X_2)^n \le \Phi_n(X) \le \frac{\left[ (1+R) X_1 + R X_2 \right]^n}{R^n}\,, \qquad X\in [0,\infty)^2\,,
	\end{equation}
where $\nu_n$ is the positive number defined by
\begin{equation*}
	\nu_n :=\exp\left\{(n-1) \left[ (1+ R\max\{1,\mu\}) \ln\Big(1+\frac{1}{R\max\{1,\mu\}}\Big) - 1 \right]\right\}-1 >0\,.
\end{equation*}
\end{lemma}

\begin{proof}
On the one hand, since the function
\begin{equation*}
	\chi (z) := \frac{\mu R + [1+R(1-\mu)] z}{\mu R + R(1-\mu) z}\,, \qquad z\in [0,1]\,,
	\end{equation*}
is increasing, we deduce from \eqref{lf01} that, for $1\le j \le n$, 
\begin{equation*}
	a_{j,n} = \binom{n}{j} \prod_{k=0}^{j-1} \chi\left( \frac{k}{n-1} \right) \le \binom{n}{j} [\chi(1)]^j = \left( \frac{1+R}{R} \right)^j \binom{n}{j}\,.
\end{equation*} 
The upper bound in \eqref{LUB} is then a straightforward consequence of the above inequality.

On the other hand, in order to estimate $\Phi_n(X)$, $X\in[0,\infty)^2$, from below we infer from \eqref{lf01a} that 
\begin{equation*}
	a_{j,n} \ge \binom{n}{j}\,, \qquad 0 \le j \le n-1\,.
\end{equation*}
	\ When estimating the coefficient $a_{n,n}$ from below we need to be more subtle and proceed as follows:
	\begin{align*}
		a_{n,n} & =  \prod_{k=0}^{n-1} \frac{k +\alpha_{k,n}}{\alpha_{k,n}}=\prod_{k=1}^{n-1} \left( 1 + \frac{k}{R [ k + \mu(n-k-1)]} \right)\\[1ex]
		& \geq \prod_{k=1}^{n-1} \left( 1 + \frac{k}{R\max\{1,\mu\}(n-1)} \right) \,.
	\end{align*}
	Now,
	\begin{align*}
		\ln\left( \prod_{k=1}^{n-1} \left( 1 + \frac{k}{R\max\{1,\mu\}(n-1)}\right) \right) & = \sum_{k=1}^{n-1} \ln\left( 1 + \frac{k}{R\max\{1,\mu\}(n-1)} \right) \\
		& \ge (n-1) \sum_{k=1}^{n-1} \int_{(k-1)/(n-1)}^{k/(n-1)} \ln\left( 1+\frac{x}{R\max\{1,\mu\}} \right)\ \mathrm{d}x \\
		& = (n-1) \int_0^1 \ln\left( 1+\frac{x}{R\max\{1,\mu\}} \right)\ \mathrm{d}x \\[1ex]
		& = (n-1) \left[ (1+ R\max\{1,\mu\}) \ln\Big(1+\frac{1}{R\max\{1,\mu\}}\Big) - 1 \right] 
	\end{align*}
	and, taking into account that 
	\begin{equation*}
	(1+x)\ln\Big(1+\frac{1}{x}\Big)>1 \quad\text{for $ x>0$}\,,
	\end{equation*}
	we end up with
	\begin{equation*}
		a_{n,n} \ge \exp\left\{(n-1) \left[ (1+ R\max\{1,\mu\}) \ln\Big(1+\frac{1}{R\max\{1,\mu\}}\Big) - 1 \right]\right\}= 1 + \nu_n\,.		
	\end{equation*}
	We thus have
	\begin{align*}
		\Phi_n(X) &\ge \nu_n X_1^n + \sum_{j=0}^{n} \binom{n}{j} X_1^j X_2^{n-j} =\nu_n X_1^n+   (X_1+X_2)^{n}\,,
	\end{align*}
	and the proof is complete. 
\end{proof}

\section{An auxiliary elliptic system}\label{secapB}

In this appendix, we establish Lemma~\ref{lem.ap1}, which is an important argument in the proof of Lemma~\ref{lem.ex2}. 
Let $\tau>0$,
 $B=(b_{jk})_{1\le j,k\le 2}\in C(\mathbb{R}^2,\mathbf{M}_2(\mathbb{R}))$, and~${A=(a_{jk})_{1\le j,k\le 2}\in \mathbf{SPD}_2(\mathbb{R})}$ 
satisfy~${AB(X)\in \mathbf{SPD}_2(\mathbb{R})}$ for all $X\in \mathbb{R}^2$ and assume that there is $\delta_1>0$ with the property
\begin{equation}
	\langle AB(X)\xi,\xi \rangle \ge \delta_1 |\xi|^2\,, \qquad (X,\xi)\in \mathbb{R}^2\times\mathbb{R}^2\,. \label{ap1}
\end{equation}
Since $A\in \mathbf{SPD}_2(\mathbb{R})$, there is also $\delta_2>0$ such that
\begin{equation}
	\langle A\xi,\xi \rangle \ge \delta_2 |\xi|^2\,, \qquad \xi\in\mathbb{R}^2\,. \label{ap2}
\end{equation}
Here, $\Omega$ is a bounded interval of $\mathbb{R}$ ($d = 1$) and we recall that, in that specific case, $H^1(\Omega)$ embeds continuously in $L_\infty(\Omega)$, so that there is $\Lambda>0$ with
\begin{equation}
	\| z\|_{\infty} \le \Lambda \|z\|_{H^1}\,, \qquad z\in H^1(\Omega)\,. \label{EHL}
\end{equation}

\begin{lemma}\label{lem.ap1}
	Given $U\in L_2(\Omega,\mathbb{R}^2)$, there is a solution $u\in H^1(\Omega,\mathbb{R}^2)$ to the nonlinear elliptic equation
	\begin{equation}
	\int_\Omega \left[ \langle u , v \rangle + \tau \langle B(u) \partial_x u , \partial_x v \rangle \right]\ \mathrm{d}x = \int_\Omega \langle U , v \rangle\ \mathrm{d}x\,, \qquad v\in H^1(\Omega,\mathbb{R}^2)\,. \label{ap3}
\end{equation}
Moreover, if 
\begin{equation}
	\begin{split}
		b_{11}(X) \geq  b_{12}(X) & = 0\,, \qquad X\in (-\infty,0)\times \mathbb{R}\,, \\
		b_{22}(X) \geq b_{21}(X) & = 0\,, \qquad X\in \mathbb{R}\times (-\infty,0)\,,
	\end{split} \label{ap10}
\end{equation}
and $U(x)\in [0,\infty)^2$ for a.a. $x\in\Omega$, then $u(x)\in [0,\infty)^2$ for a.a. $x\in\Omega$.
\end{lemma}

\begin{proof}
To set up a fixed point scheme, we define $\delta_0:=\min\{\tau \delta_1, \delta_2\}$ and introduce the compact and convex subset $\mathcal{K}$ of $L_2(\Omega,\mathbb{R}^2)$ defined by
\begin{equation}
	\mathcal{K} := \left\{ u\in H^1(\Omega,\mathbb{R}^2)\ :\ \|u\|_{H^1} \le \frac{\|AU\|_2}{\delta_0}\right\}\,, \label{ap100}
\end{equation}		
the compactness of $\mathcal{K}$ being a straightforward consequence of the compactness of the embedding of~${H^1(\Omega,\mathbb{R}^2)}$ in $L_2(\Omega,\mathbb{R}^2)$.
 According to \eqref{EHL}, 
\begin{equation}
	\|u\|_\infty \le \frac{\Lambda\|AU\|_2}{\delta_0}\,, \qquad u\in\mathcal{K}\,. \label{BU}
\end{equation}
	
We now consider $u\in \mathcal{K}$ and define a  bilinear form $b_u$ on~${H^1(\Omega,\mathbb{R}^2)}$ by
	\begin{equation*}
		b_u(v,w) := \int_\Omega \left[ \langle Av , w \rangle + \tau \langle AB(u) \partial_x v , \partial_x w \rangle \right]\ \mathrm{d}x\,, \qquad (v,w)\in H^1(\Omega,\mathbb{R}^2)\times H^1(\Omega,\mathbb{R}^2) \,.
	\end{equation*} 
Owing to \eqref{ap1} and \eqref{ap2},
\begin{equation}
	b_u(v,v) \ge \delta_0 \|v\|_{H^1}^2\,, \qquad v\in H^1(\Omega,\mathbb{R}^2)\,, \label{ap4}
\end{equation}
while the  continuity of $B$ and the boundedness \eqref{BU} of $u$ guarantee  that
\begin{equation*}
	|b_u(v,w)| \le b_u^* \|v\|_{H^1} \|w\|_{H^1}\,, \qquad (v,w)\in H^1(\Omega,\mathbb{R}^2)\times H^1(\Omega,\mathbb{R}^2) \,,
\end{equation*}
where
\begin{equation*}
	b_u^* := 2\max_{1\le j,k\le 2}\{|a_{jk}|\} \left( 1 + 2\tau \max_{1\le j,k\le 2}\{\| b_{jk}(u) \|_\infty\} \right)\,.
\end{equation*}
We then infer from Lax-Milgram's theorem that there is a unique $\mathcal{V}[u]\in H^1(\Omega,\mathbb{R}^2)$ such that
\begin{equation}
	b_u(\mathcal{V}[u],w) = \int_\Omega \langle AU , w \rangle\ \mathrm{d}x\,, \qquad w\in H^1(\Omega,\mathbb{R}^2)\,. \label{ap5}
\end{equation}
In particular, taking $w=\mathcal{V}[u]$ in \eqref{ap5} and using \eqref{ap4} and H\"older's inequality give
\begin{equation*}
	\delta_0 \|\mathcal{V}[u]\|_{H^1}^2 \le b_u(\mathcal{V}[u],\mathcal{V}[u]) \le \|AU\|_2 \|\mathcal{V}[u]\|_2 \le \|AU\|_2 \|\mathcal{V}[u]\|_{H^1}\,.
\end{equation*}
 Consequently, 
\begin{equation}
	\|\mathcal{V}[u]\|_{H^1} \le \frac{\|AU\|_2}{\delta_0}  \;\;\text{ and }\;\; \mathcal{V}[u]\in\mathcal{K} \,. \label{ap6}
\end{equation}

We now claim that the map $\mathcal{V}$ is continuous  on $\mathcal{K}$ with respect to the norm-topology of $L_2(\Omega,\mathbb{R}^2)$. 
Indeed, consider a sequence  $(u_j)_{j\ge 1}$ in  $\mathcal{K}$ and $u\in \mathcal{K}$  such that 
\begin{equation*}
	\lim_{j\to\infty} \|u_j-u\|_2=0\,.
\end{equation*}
Upon extracting a subsequence (not relabeled), we may assume that
\begin{equation*}
	\lim_{j\to\infty} u_j(x) = u(x) \;\text{ for a.a. }\; x\in \Omega\,,
\end{equation*}
so that the continuity of $B$ and \eqref{BU} ensure that
\begin{subequations}\label{ap7}
\begin{equation}
	\lim_{j\to\infty} B(u_j(x)) = B(u(x))  \;\text{ for a.a. }\; x\in \Omega \label{ap7.1}
\end{equation}
and
\begin{equation}
	\max\left\{ \|B(u)\|_\infty , \sup_{j\ge 1}\{\|B(u_j)\|_\infty\} \right\} \le \max_{|X|\le \Lambda \|AU\|_2/\delta_0}\{|B(X)|\}\,. \label{ap7.2}
\end{equation}
\end{subequations}
It also follows from \eqref{ap6} and the compactness of the embedding of $H^1(\Omega,\mathbb{R}^2)$ in $L_2(\Omega,\mathbb{R}^2)$ that there is $v\in H^1(\Omega,\mathbb{R}^2)$ such that, after possibly extracting a further subsequence, 
\begin{equation}
	\lim_{j\to\infty} \|\mathcal{V}[u_j] - v  \|_2 = 0 \qquad\text{ and }\qquad \mathcal{V}[u_j] \rightharpoonup v \;\text{ in }\; H^1(\Omega,\mathbb{R}^2)\,. \label{ap8}
\end{equation}
 Since 
\begin{equation*}	
	\int_\Omega \langle AB(u_j)\partial_x \mathcal{V}[u_j] , \partial_x w \rangle\ \mathrm{d}x = \int_\Omega \langle \partial_x \mathcal{V}[u_j], AB(u_j)\partial_x w \rangle\ \mathrm{d}x\,,
\end{equation*}
due to the symmetry of $AB(X)$ for $X\in\mathbb{R}^2$, it  readily follows from \eqref{ap7}, \eqref{ap8},  and Lebesgue's dominated convergence theorem that we may pass to the limit $j\to\infty$ in the variational identity~\eqref{ap5} for $\mathcal{V}[u_j]$ and conclude that
\begin{equation*}
	b_u(v,w) = \int_\Omega \langle AU , w \rangle\ \mathrm{d}x\,, \qquad w\in H^1(\Omega,\mathbb{R}^2)\,,
\end{equation*}
that is, $v=\mathcal{V}[u]$. We have thus shown that any subsequence of  $(\mathcal{V}[u_j])_{j\ge 1}$ has a  subsequence that converges to $\mathcal{V}[u]$, which proves the claimed continuity of the map $\mathcal{V}$.  We are therefore in a position to apply Schauder's fixed point theorem, see \cite[Theorem~11.1]{GT2001} for instance, and conclude that the map $\mathcal{V}$ has a fixed point $u\in \mathcal{K}$.
 In particular, the function~$u$ satisfies
\begin{equation*}
	b_u(u,w) = \int_\Omega \langle AU , w \rangle\ \mathrm{d}x\,, \qquad w\in H^1(\Omega,\mathbb{R}^2)\,.
\end{equation*}
Now, given $v\in H^1(\Omega,\mathbb{R}^2)$, the function $w=A^{-1}v$ also belongs to $H^1(\Omega,\mathbb{R}^2)$ and we infer from the above identity and the symmetry of $A$ that
\begin{align*}
	\int_\Omega \langle U , v \rangle\ \mathrm{d}x & = \int_\Omega \langle AU , w \rangle\ \mathrm{d}x = b_u(u,w) = b_u(u,A^{-1}v) \\
	&  = \int_\Omega \left[ \langle u , v \rangle + \tau \langle B(u) \partial_x u , \partial_x v \rangle \right]\ \mathrm{d}x \,.
\end{align*}
We have thus constructed a   solution $u\in H^1(\Omega,\mathbb{R}^2)$ to \eqref{ap3}.

We now turn to the non-negativity-preserving property and assume that $U(x)\in [0,\infty)^2$ for a.a.~${x\in\Omega}$. 
Let $u\in H^1(\Omega,\mathbb{R}^2)$ be a   solution to \eqref{ap3} and set~$\varphi:=-u$. Then $(\varphi_{1,+},\varphi_{2,+})$ belongs to $H^1(\Omega,\mathbb{R}^2)$ and it follows from \eqref{ap3} that
\begin{align}
	& \int_\Omega \left( \varphi_1 \varphi_{1,+}  + \varphi_2 \varphi_{2,+} + \tau \sum_{j,k=1}^2 b_{jk}(u) \partial_x \varphi_k \partial_x (\varphi_{j,+}) \right)\ \mathrm{d}x \nonumber \\
	& \hspace{5cm} = - \int_\Omega \left( U_1 \varphi_{1,+}  + U_2 \varphi_{2,+} \right)\ \mathrm{d}x \le 0\,. \label{ap11}
\end{align}
We now infer from \eqref{ap10} that 
\begin{align*}
	b_{11}(u) \partial_x\varphi_1 \partial_x\varphi_{1,+} & = b_{11}(u) \mathbf{1}_{(-\infty,0)}(u_1) |\partial_x u_1|^2 \ge 0\,, \\
	b_{12}(u) \partial_x\varphi_2 \partial_x\varphi_{1,+} & = b_{12}(u) \mathbf{1}_{(-\infty,0)}(u_1) \partial_x u_1 \partial_x u_2 = 0\,, \\
	b_{21}(u) \partial_x\varphi_1 \partial_x\varphi_{2,+} & = b_{21}(u) \mathbf{1}_{(-\infty,0)}(u_2) \partial_x u_1 \partial_x u_2 = 0\,, \\
	b_{22}(u) \partial_x\varphi_2 \partial_x\varphi_{2,+} & = b_{22}(u) \mathbf{1}_{(-\infty,0)}(u_2) |\partial_x u_2|^2 \ge 0\,,
\end{align*}
so that the second term on the left-hand side of \eqref{ap11} is non-negative. Consequently, \eqref{ap11} gives
\begin{equation*}
	\int_\Omega \left( |\varphi_{1,+}|^2  + |\varphi_{2,+}|^2 \right)\ \mathrm{d}x \le 0\,, 
\end{equation*}
which implies that $\varphi_{1,+}=\varphi_{2,+}=0$ a.e. in $\Omega$. Hence, $u(x)\in [0,\infty)^2$ for a.a. $x\in\Omega$ and the proof of Lemma~\ref{lem.ap1} is complete.
\end{proof}

\bibliographystyle{siam}
\bibliography{UnifboundTFMP}

\begin{thebibliography}{10}

\bibitem{ACCL2019}
{\sc A.~{Ait Hammou Oulhaj}, C.~{Canc\`es}, C.~{Chainais-Hillairet}, and
  {\relax Ph}.~{Lauren\c{c}ot}}, {\em {Large time behavior of a two phase
  extension of the porous medium equation}}, {Interfaces Free Bound.}, 21
  (2019), pp.~199--229.

\bibitem{AIJM2018}
{\sc J.~{Alkhayal}, S.~{Issa}, M.~{Jazar}, and R.~{Monneau}}, {\em {Existence
  result for degenerate cross-diffusion system with application to seawater
  intrusion}}, {ESAIM, Control Optim. Calc. Var.}, 24 (2018), pp.~1735--1758.

\bibitem{BGB2019}
{\sc G.~{Bruell} and R.~{Granero-Belinch\'on}}, {\em {On the thin film Muskat
  and the thin film Stokes equations}}, {J. Math. Fluid Mech.}, 21 (2019),
  p.~31.
\newblock Id/No 33.

\bibitem{DGJ1997}
{\sc P.~{Degond}, S.~{G\'enieys}, and A.~{J\"ungel}}, {\em {Symmetrization and
  entropy inequality for general diffusion equations}}, {C. R. Acad. Sci.,
  Paris, S\'er. I, Math.}, 325 (1997), pp.~963--968.

\bibitem{DJ2012}
{\sc M.~{Dreher} and A.~{J\"ungel}}, {\em {Compact families of piecewise
  constant functions in \(L^p (0,T;B)\)}}, {Nonlinear Anal., Theory Methods
  Appl., Ser. A, Theory Methods}, 75 (2012), pp.~3072--3077.

\bibitem{ELM2011}
{\sc J.~{Escher}, {\relax Ph}.~{Lauren\c{c}ot}, and B.-V. {Matioc}}, {\em
  {Existence and stability of weak solutions for a degenerate parabolic system
  modelling two-phase flows in porous media}}, {Ann. Inst. Henri Poincar\'e,
  Anal. Non Lin\'eaire}, 28 (2011), pp.~583--598.

\bibitem{EMM2012}
{\sc J.~{Escher}, A.-V. {Matioc}, and B.-V. {Matioc}}, {\em {Modelling and
  analysis of the Muskat problem for thin fluid layers}}, {J. Math. Fluid
  Mech.}, 14 (2012), pp.~267--277.

\bibitem{GT2001}
{\sc D.~{Gilbarg} and N.~S. {Trudinger}}, {\em {Elliptic partial differential
  equations of second order}}, Berlin: Springer, 2001.

\bibitem{JM2014}
{\sc M.~{Jazar} and R.~{Monneau}}, {\em {Derivation of seawater intrusion
  models by formal asymptotics}}, {SIAM J. Appl. Math.}, 74 (2014),
  pp.~1152--1173.

\bibitem{La2017}
{\sc M.~{Lambacher}}, {\em {Existence and long time asymptotics of solutions to
  a Muskat problem with multiple components}}, 2017.
\newblock {Master's Thesis, Technische Universit\"at M\"unchen}.

\bibitem{LM2013}
{\sc {\relax Ph}.~{Lauren\c{c}ot} and B.-V. {Matioc}}, {\em {A gradient flow
  approach to a thin film approximation of the Muskat problem}}, {Calc. Var.
  Partial Differ. Equ.}, 47 (2013), pp.~319--341.

\bibitem{LM2017}
\leavevmode\vrule height 2pt depth -1.6pt width 23pt, {\em {Finite speed of
  propagation and waiting time for a thin-film Muskat problem}}, {Proc. R. Soc.
  Edinb., Sect. A, Math.}, 147 (2017), pp.~813--830.

\bibitem{LM2021a}
\leavevmode\vrule height 2pt depth -1.6pt width 23pt, {\em {The porous medium
  equation as a singular limit of the thin film Muskat problem}},  (2021).
\newblock arXiv:2108.09032.

\bibitem{WM2000}
{\sc A.~W. {Woods} and R.~{Mason}}, {\em {The dynamics of two-layer
  gravity-driven flows in permeable rock}}, {J. Fluid Mech.}, 421 (2000),
  pp.~83--114.

\bibitem{ZM2015}
{\sc J.~{Zinsl} and D.~{Matthes}}, {\em {Transport distances and geodesic
  convexity for systems of degenerate diffusion equations}}, {Calc. Var.
  Partial Differ. Equ.}, 54 (2015), pp.~3397--3438.

\end{thebibliography}

\end{document}